\newif{\ifdraft}\drafttrue
\draftfalse

\newif{\ifarXiv}\arXivtrue

\ifarXiv
\documentclass[final,a4paper]{elsarticle}
\else
\ifdraft
\documentclass[a4paper]{elsarticle}
\usepackage{showkeys}
\else
\documentclass[final]{elsarticle}
\fi 
\fi

\ifarXiv
\makeatletter
\def\ps@pprintTitle{%
	\let\@oddhead\@empty
	\let\@evenhead\@empty
	\let\@oddfoot\@empty
	\let\@evenfoot\@oddfoot
}
\makeatother
\fi

\usepackage[utf8]{inputenc}
\usepackage{amssymb,amsmath,enumerate}
\usepackage{fancyhdr}
\usepackage{stmaryrd}
\usepackage{xspace}
\usepackage{tikz}
\usetikzlibrary{trees, automata}
\usepackage{multicol}
 \usepackage{amsthm}
 \ifarXiv
 \usepackage[left=3.5cm,right=3.5cm,top=3.2cm,bottom=4.5cm]{geometry}
 \fi
  \usepackage{microtype}
 
\newtheorem{theorem}{Theorem}
\newtheorem{lemma}[theorem]{Lemma}
\newtheorem{proposition}[theorem]{Proposition}
\newtheorem{corollary}[theorem]{Corollary}
\newtheorem{lemma*}{Lemma}

\theoremstyle{definition}
\newtheorem{definition}[theorem]{Definition}
\newtheorem{remark}[theorem]{Remark}
\newtheorem{example}[theorem]{Example}

\newtheorem{thmy}{Theorem}

\renewenvironment{proof}[1][Proof]{
\begin{trivlist}
\item[\hskip \labelsep {\bfseries #1.}]}{\hspace*{\fill}$\Box$\end{trivlist}
}

\usepackage{amsfonts}
\usepackage{amssymb}
\usepackage{amsmath}

\usepackage{tabularx,hhline,rotating,enumitem,xspace}
\usepackage{tikz}

\usepackage{mathdots}

\usepackage{prettyref}

\renewcommand{\setminus}{\mysetminus}

\newenvironment{vd}{\noindent\color{blue} VD : }{}

\newenvironment{AM}{\noindent\color{red} AM : }{}

\newenvironment{aw}{\noindent\color{magenta} AW : }{}

 \newcommand{\prref}\prettyref
\newrefformat{thm}{Theorem~\ref{#1}}
\newrefformat{lem}{Lem\-ma~\ref{#1}}
\newrefformat{def}{Definition~\ref{#1}}
\newrefformat{cor}{Corollary~\ref{#1}} 
\newrefformat{prop}{Proposition~\ref{#1}}
\newrefformat{sec}{Section~\ref{#1}}
\newrefformat{cap}{Chapter~\ref{#1}}
\newrefformat{bsp}{Example~\ref{#1}}
\newrefformat{rem}{Remark~\ref{#1}}
\newrefformat{fig}{Figure~\ref{#1}}
\newrefformat{eq}{(\ref{#1})}
\newrefformat{ex}{Example~\ref{#1}}
\newrefformat{tab}{Table~\ref{#1}}
\newrefformat{app}{Appendix~\ref{#1}}
\newrefformat{alg}{Algorithm~\ref{#1}}

\newcommand{\mysetminusD}{\raisebox{.8pt}{\hbox{\tikz{\draw[line width=0.6pt,line cap=round] (3.5pt,0pt) -- (0,5.2pt);}}}}
\newcommand{\mysetminusT}{\mysetminusD}
\newcommand{\mysetminusS}{\raisebox{.5pt}{\hbox{\tikz{\draw[line width=0.45pt,line cap=round] (2.2pt,0) -- (0,3.8pt);}}}}
\newcommand{\mysetminusSS}{\raisebox{.35pt}{\hbox{\tikz{\draw[line width=0.4pt,line cap=round] (1.5pt,0) -- (0,2.8pt);}}}}

\newcommand{\mysetminus}{\mathbin{\mathchoice{\mysetminusD}{\mysetminusT}{\mysetminusS}{\mysetminusSS}}}

\newcommand{\prob}{probability\xspace}

\newcommand{\set}[2]{\left\{#1\; \middle|\; #2\right\}}

\newcommand{\oneset}[1]{\left\{\mathinner{#1}\right\}}
\newcommand{\os}{\oneset}

\newcommand{\abs}[1]{\left|\mathinner{#1}\right|}
\newcommand{\Abs}[1]{\left\Vert\mathinner{#1}\right\Vert}

\newcommand{\ceil}[1]{\left\lceil\mathinner{#1} \right\rceil}

\newcommand{\gen}[1]{\left< \mathinner{#1} \right>}
\newcommand{\genr}[2]{\left< \, \mathinner{#1}\; \middle|\;\mathinner{#2} \, \right>}

\newcommand{\scalp}[2]{\langle {#1},\,{#2} \rangle}

\newcommand{\N}{\ensuremath{\mathbb{N}}}
\newcommand{\Z}{\ensuremath{\mathbb{Z}}}

\newcommand{\R}{\ensuremath{\mathbb{R}}}
\newcommand{\C}{\ensuremath{\mathbb{C}}}
\newcommand{\B}{\ensuremath{\mathbb{B}}}

\renewcommand{\phi}{\varphi}
\newcommand{\eps}{\varepsilon}

\newcommand{\gam}{\gamma}
\newcommand{\del}{\delta}

\newcommand{\Sig}{\Sigma}

\newcommand{\Del}{\Delta}
\newcommand{\Lam}{\Lambda}

\newcommand\GG{\Gamma}

\newcommand\OO{\Omega}

\newcommand{\Oh}{\mathcal{O}}

\newcommand{\cA}{\mathcal{A}}

\newcommand{\cL}{\mathcal{L}}

\newcommand{\cP}{\mathcal{P}}
\newcommand{\cN}{\mathcal{N}}

\newcommand{\WP}{word problem\xspace}
\newcommand{\CP}{conjugacy problem\xspace}

\newcommand{\BS}[2]{\ensuremath{\mathrm{\bf{BS}}_{#1,#2}}\xspace}

\newcommand{\BG}{\ensuremath{\mathrm{\bf{G}}_{1,2}}\xspace 
}

\newcommand{\smalloverline}[1]
{{\mspace{.8mu}\overline{\mspace{-.8mu}#1\mspace{-.8mu}}\mspace{.8mu}}}
\newcommand{\ov}[1]{\smalloverline{#1}}
\newcommand{\oi}[1]{{#1}^{-1}}

\newcommand{\oT}[1]{{#1}^{T}}

\newcommand{\lquot}[2]{{#1}\backslash {#2}}
\newcommand{\rquot}[2]{{#1} /\!\;\!{#2}}

\newcommand{\lQuot}[2]{{#1}\backslash {#2}}

\newcommand{\IFF}{if and only if\xspace}

\newcommand{\sse}{\subseteq}

\newcommand{\sm}{\setminus}

\newcommand\ei[1]{{\emph{#1}\xspace}\index{#1}}

\newcommand\ie{i.\,e., }

\newcommand\eg{e.\,g.\xspace}

\newcommand{\tar}{\tau}
\newcommand{\sor}{\iota}

\newcommand{\stl}{t}
\newcommand{\y}{t}

\begin{document}

\begin{frontmatter}
	
	

		\title{Amenability of Schreier graphs and strongly generic algorithms for the conjugacy problem} 
	
	
	\author[fmi]{ Volker Diekert}
	\author[stevens]{ Alexei G.\ Myasnikov}
	\author[stevens]{ Armin Wei\ss\footnote{This research was conducted while the third author was working at FMI, Universit{\"a}t Stuttgart, Germany.}}
	
	\address[fmi]{FMI, Universit{\"a}t Stuttgart, Germany}
	\address[stevens]{Stevens Institute of Technology, Hoboken, NJ, USA}
	
	\begin{abstract}
		In various occasions the conjugacy problem in finitely generated amalgamated products and HNN extensions can be decided efficiently for elements which cannot be conjugated into the base groups. Thus, the question arises ``how many'' such elements there are. This question can be formalized using the notion of strongly generic sets and lower bounds can be proven by applying the theory of amenable graphs:
		
		In this work we examine Schreier graphs of amalgamated products and HNN extensions. For an amalgamated product $G =H\star_AK $ with $[H:A]\geq [K:A]\geq 2$, the Schreier graph with respect to $H$ or $K$ turns out to be non-amenable if and only if $[H:A]\geq 3$. Moreover, for an HNN extension of the form $G = \left< \, H,b \;\middle| \; bab^{-1}=\phi(a), a \in A \, \right>$, we show that the Schreier graph of $G$ with respect to the subgroup $H$ is non-amenable if and only if $A\neq H \neq \phi(A)$.
		
		As application of these characterizations we show that the conjugacy problem in fundamental groups of finite graphs of groups with  finitely generated free abelian vertex groups can be solved in polynomial time on a strongly generic set. Furthermore, the conjugacy problem in groups with more than one end can be solved with a strongly generic algorithm which has essentially the same time complexity as the word problem. These are rather striking results as the word problem might be easy, but the conjugacy problem might be even undecidable. Finally, our results yield a new proof that the set where the conjugacy problem of the Baumslag group $\mathrm{\bf{G}}_{1,2}$ is decidable in polynomial time is also strongly generic.
		
	\end{abstract}
	
	\begin{keyword}
%
%
		generic case complexity \sep 
		amenability \sep 
		Schreier graph \sep 
		HNN extension \sep 
		amalgamated product \sep 
		conjugacy problem
		
		\MSC 20F65 \sep 05C81 \sep 20E06
	\end{keyword}
	
\end{frontmatter}


\section{Introduction}\label{sec:intro} 
The conjugacy problem of a group $G$ asks on input of two words $x,y$ over some set of generators whether there exists some group element $z$ such that $z x z^{-1} = y$ as equality in $G$. In recent years, conjugacy played an important role in non-commutative cryptography, see \eg~\cite{CravenJ12,GrigorievS09,SZ1}. These applications for public-key cryptosystems are based on the idea that it is easy to create elements which are conjugated, but to check whether two given elements are conjugated might be difficult. 

For cryptographic applications not only the existence of difficult instances matters, but for the generation of keys it is important to find these efficiently. Ideally, a random instance is difficult with sufficiently high probability. The complexity of random instances is studied via generic complexity, a concept introduced by Kapovich, Myasnikov, Spilrain and Schupp in \cite{KMSS1}. Some set is called (strongly) generic if the portion of words outside the set tends (exponentially fast) to zero with increasing length of the words.
Kapovich et al.\ showed that several algorithmic problems in group theory can be decided efficiently by algorithms which give an answer on generic or even strongly generic sets. In particular, the conjugacy problem can be decided generically (but not strongly generically!) in linear time as soon as the group admits an infinite cyclic quotient \cite{KMSS1}. On the other hand, there are problems for which there is no strongly generic algorithm (\eg, the halting problem \cite{MiasnikovR08}) or for which no generic polynomial time algorithm is known such as the problem of integer factorization.

The word problem of a group asks whether some word over the generators represents the identity of the group. It is a special case of the conjugacy problem since some group element is conjugate to the identity if and only if it is the identity.
However, the conjugacy problem is inherently more difficult than the word problem. Miller's group \cite{Miller1} is a famous example for that: its \WP is solvable in polynomial time (actually in logspace), but the conjugacy problem is undecidable. In \cite{BorovikMR07} Borovik, Myasnikov, and Remeslennikov showed that, nevertheless, the conjugacy problem is decidable in polynomial time on a strongly generic set. This means essentially that for ``random inputs'' conjugacy is easy to decide.

Bogopolski, Martino, and Ventura constructed another example in \cite{BogopolskiMV10} with an easy word problem but undecidable conjugacy problem: an HNN extension of $\Z^4$ with several stable letters, which is actually a $\Z^4$-by-free group. For this group the decidability of the word problem is immediate by the standard algorithm for HNN extensions, but, again, the conjugacy problem is undecidable. Our results show that the conjugacy problem is strongly generically decidable in polynomial time, see \prettyref{cor:freeamen}.

Even if the \CP is decidable, there might be a non-elementary gap between the complexities. Perhaps, one of the most striking examples so far (for a not-on-purpose construction) is the Baumslag group (or Baumslag-Gersten group) $\BG = \genr{a,t,b}{b ab^{-1}=t,t a t^{-1} = a^2}$. 
It is an HNN extension with stable letter $b$ of the structural much simpler Baumslag-Solitar group $\BS12= \genr{a,t}{t a t^{-1} = a^2}$.
The Baumslag group has a non-elementary Dehn function \cite{G1} and it was a prominent candidate for having the most difficult word problem among all one-relator groups until Myasnikov, Ushakov, and Won showed in \cite{muw11bg} that its word problem is solvable in polynomial time! However, there are strong indications that this does not transfer to the conjugacy problem for the group $\BG$. We conjectured in \cite{DiekertMW14} that the conjugacy problem for $\BG$ is non-elementary on average. Nevertheless, having a non-elementary time complexity on average does not prevent the set of ``difficult instances'' to be extremely sparse: we could use the techniques developed in \cite{muw11bg} to design a polynomial time algorithm for the conjugacy problem which works for elements which cannot be conjugated into $\BS12$. Moreover, we showed that this is a strongly generic set by deriving some explicit bounds; thus, we established a strongly generic polynomial time algorithm to solve the \CP for $\BG$ \cite{DiekertMW14}.

A crucial tool when dealing with the conjugacy problem in HNN groups is Collins' Lemma, which distinguishes two cases: the first case is that the input words can be conjugated into the base group of the HNN extension, the second is that they cannot. In the above examples (except Miller's group), the second case can be solved efficiently whereas the first case contains the difficulty. Therefore, the aim of this work is to provide bounds for the portion of inputs which fall into the second case by showing that this second case is strongly generic. 
\medskip

\noindent {\bf Outline.} We consider any finitely generated group $G$ 
which is either an amalgamated product
$G =H\star_A K$ with $H\neq A \neq K$ or an HNN extension $G = \genr{H,\y}{\y a\y^{-1}=\phi(a) \text{ for } a \in A}$ with stable letter $t$ and an isomorphism 
$\phi: A \to B$ for subgroups $A$ and $B$ of $H$.
We characterize precisely when the Schreier graph $\GG(G,P,\Sig)$ is non-amenable, 
see \prref{thm:a} and \prref{thm:h}.  

The notion of strongly generic sets is closely related to non-amenability: The Schreier graph $\GG(G,P,\Sig)$ is non-amenable if and only if the set of words which do not represent an element of the subgroup $P$ is strongly generic. Using this characterization, we derive in \prettyref{thm:one} that under certain conditions the words which cannot be conjugated into the base groups of the HNN extension (resp.\ amalgamated product) form a strongly generic set.

This yields another proof for the result of \cite{DiekertMW14} that the set of words in the Baumslag group \BG which cannot be conjugated into the subgroup \BS12 is strongly generic. We present two more applications of our results: 
 First, we show in \prettyref{cor:freeamen} that the conjugacy problem of the $\Z^4$-by-free group from \cite{BogopolskiMV10}~-- and, more generally, of all fundamental groups of finite graphs of groups with finitely generated free abelian vertex groups~-- is decidable in polynomial time on a strongly generic set.
 
 The second application is about finitely generated groups with more than one end. These groups have a characterization as amalgamated product or HNN extension over some finite subgroup. We show that in this case in a strongly generic setting the \CP has essentially the same difficulty as the \WP, see \prref{cor:infendgeneric}. At first glance, this result is quite surprising because the \WP in $G$ can be easy and the \CP can be undecidable. However, \prref{cor:infendgeneric} affirms that in practice we might spend a hard time to find difficult instances for the conjugacy problem at all. 
 
 The paper is organized as follows: after introducing some notation, we state our main results and consider their algorithmic applications in \prref{sec:results}. Then, we review the concept of amenability in \prref{sec:amenability} and present the proofs of our main results in \prref{sec:repr}.
This work is an extended version of the conference paper \cite{DiekertMW15amenability}. It contains proofs of all results and provides some additional examples and explanatory remarks. Moreover, we introduce a slightly more general statement of \prref{cor:freeamen}.
 
 \section{Notation}\label{sec:notation}
\noindent {\bf Words.} An \emph{alphabet} is a (finite) set $\Sig$; an element $a \in \Sig$ is called a \emph{letter}. 
The free monoid over $\Sig$ is denoted by $\Sig^*$, its elements
are called {\em words}. The length of a word $w$
is denoted by $\abs w$, and $\Sig^n$
forms the set of words of length $n$. The empty word is denoted by $1$. 
Let $a \in \Sig$ be a letter and $w\in \Sig^*$; the number of occurrences of $a$ in $w$ is denoted by ${\abs w}_a$. 
If $w,p,x,q$ are words such that $w = pxq$, then we call $p$ a \emph{prefix}, $x$ a \emph{factor}, and $q$ a \emph{suffix} of $w$. We also say that $w = uxv$ is a \emph{factorization}.

\medskip
\noindent {\bf Involutions.} An involution on a set $S$ is a mapping 
$x \mapsto \ov x$ such that $\ov {\ov x} = x$. If $S$ is a semigroup (in particular, if $S=\Sig^*$), then we additionally demand that $\ov{xy} = \ov y \, \ov x$. 
Usually, we consider fixed-point-free involution, \ie $x \neq \ov x$.

\medskip
\noindent {\bf Groups.} 
We consider groups $G$ together with finite sets of monoid generators $\Sig$. 
	That is there is a surjective homomorphism  $\eta:\Sig^* \to G$ (a \emph{monoid presentation}). 
	In order to keep notation simple, we usually suppress the homomorphism $\eta$ and consider words also as group elements. We write $w=_{G}w'$ as a shorthand of 
	 $\eta(w)=\eta(w')$. Thus, $w=_{G}w'$ means that $w$ and $w'$ represent the same element in the group $G$. We extend this notation and write  $ w \in_G A$ (resp.\ $L \sse_G A$) instead of $\eta(w) \in A$ (resp.\ $\eta(\Sig) \sse A$) for $A \sse G$, $w \in \Sig^*$, $L\sse \Sig^*$. 
	
If there is a fixed-point-free involution $\ov{\,\cdot\,}: \Sig \to \Sig$ satisfying $a \mapsto \ov a =_G \oi a$, then we call $\Sig$ {\em symmetric}. This means for every letter $a\in \Sig$ there is a \emph{formal inverse} $\ov a \neq a$.

If $\Lam$ generates $G$ as a group, we can add formal inverses and obtain a symmetric generating set $\Sig = \Lam \cup \ov \Lam$.

Let $\Sig$ be symmetric and let $w \in \Sig^*$.
We say that $w$ is \emph{reduced} if there is no factor $a\ov a$ for any letter $a \in \Sig$. 
It is called \emph{cyclically reduced} if $ww$ is reduced. There is a canonical bijection between the free group $F(\Lam)$ and the set of reduced words over $\Sig = \Lam \cup \ov \Lam$.
For words (or group elements) we write $x \sim_G y$ to denote conjugacy, \ie $x \sim_G y$ if and only if there exists some $z\in G$ such that $zx \oi z =_G y$.

The paper is about finitely generated groups which are either amalgamated products or HNN extensions.
A group $G$ is an amalgamated product if $$G =H\star_A K=\genr{H,K}{\phi(a)=\psi(a) \text{ for }a\in A}$$
for groups $H$ and $K$ with a common subgroup $A$ where $\phi$ and $\psi$ are the embeddings of $A$ in $H$ and $K$. An HNN extension is of the form 
$$G = \genr{H,\y}{\y a\y^{-1}=\phi(a) \text{ for } a \in A}$$ with a stable letter $t$ and an isomorphism $\phi: A \to B$ for subgroups $A$ and $B$ of $H$. 

According to the more general notion of \emph{graph of groups} \cite{serre80}, we refer to the base groups 
 $H,K$ as ``vertex groups'' and to $A$ as ``edge group''.
 Elements of $G$ which are conjugate to some element in one of the vertex groups are called {\em elliptic}, the others are called 
{\em hyperbolic}\footnote{The distinction between elliptic and hyperbolic elements stems from 
group actions on trees. The group $G$ acts naturally on a tree: its Bass-Serre tree corresponding to the splitting. The ``elliptic'' 
elements of $G$ are those which fix a vertex of the tree. These are 
in turn exactly those elements which are conjugates of elements in $H$ or $K$. The ``hyperbolic'' 
elements are those which act without fixed points.}. Thus, if $G$ is an amalgamated product, then the
set of elliptic elements is $\bigcup_{g\in G}g(H\cup K)g^{-1}$; if $G$ is an HNN extension, then the
set of elliptic elements is 
$\bigcup_{g\in G}gHg^{-1}$. With $[H:A]$ we denote the index of the subgroup $A$ in $H$.

Let $G=H\star_A K$ be an amalgamated product and let $\Sig_H$ and $\Sig_K$ be finite generating sets of $H$ and $K$. We define a finite  generating set of $G$ as
$\Sig = \Sig_H \cup \Sig_K$.
A word over $\Sig$ is called \emph{Britton reduced} if the number of alternations between letters from $H$ and $K$ is minimal among all words representing the same group element.
Britton reduced words can be computed by \emph{Britton reductions}: if there is some factor $v \in \Sig_H^*\Sig_K^*\Sig_H^*$ (resp.\ $v \in \Sig_K^*\Sig_H^*\Sig_K^*$) with $v=_Gv'$ for some $v' \in \Sig_H^*$ (resp.\ $v' \in \Sig_K^*$), replace $v$ by $v'$.

If $G=\genr{H,\y}{\y a\y^{-1}=\phi(a) \text{ for } a \in A}$ is an HNN extension and $\Sig_H$ a finite generating set of $H$, we obtain generating set of $G$ as 
 $\Sig = \Sig_H \cup \oneset{t, \ov t}$.
Now, a word over $\Sig$ is called \emph{Britton reduced} if the number of letters from $\oneset{t, \ov t}$ is minimal among all words representing the same group element.
Again, Britton reduced words can be computed  by the application of \emph{Britton reductions}: if there is some factor $w=tv\ov{t}$ or  $w=\ov{t}vt$ with $v \in \Sig_H^*$ and $w=_Gw'$ for some $w' \in \Sig_H^*$, replace $w$ by $w'$. 

In both cases a word $w$ is called \emph{cyclically Britton reduced} if $ww$ is Britton reduced.

\medskip
\noindent {\bf Graphs.} 
For the notation of graphs we follow Serre's book \cite{serre80}. 
A \ei{directed graph} $\Gamma = (V,E,\sor,\tar)$ is given by the following data: 
A set of \emph{vertices}\index{$V(\cdot)$} $V= V(\Gamma)$ and a set of \emph{edges}\index{$E(\cdot)$} $E=E(\GG)$ together with two mappings\index{$\sor(\cdot)$}\index{$\tar(\cdot)$} $\sor,\tar:E \to V$. 
The vertex $\sor(e)$ is the \emph{initial} vertex (or \emph{source}) of $e$, and $\tar(e)$ is the \emph{terminal} vertex (or \emph{target}) of $e$. 
If $\tar(e) = u$ (resp.\ $\sor(e) = u$), we call $e$ an \emph{incoming edge} (resp.\ \emph{outgoing edge}) of $u$. 
The \ei{in-degree} (resp.\ \ei{out-degree}) of a vertex is the number of incoming edges (resp.\ outgoing edges); $\Gamma$ is called \ei{locally finite} if the in-degrees and out-degrees of all vertices are finite.
If both the in-degrees and out-degrees of all vertices are equal to some constant, then $\GG$ is called \emph{regular}; and if the degree is $d$, then it is \emph{$d$-regular}. 

An \ei{undirected graph} is a directed graph $\GG$ such that the set of edges $E$ is equipped with an involution $e \mapsto \ov e$ without fixed points such that $\sor(e) = \tar(\ov e)$ for all $e \in E$. In particular, 
we have $\ov{\ov e} = e$ and $\ov e \neq e$ for all $e \in E$. 
Every undirected graph is also a directed graph by forgetting the involution.

For simplicity of notation, we often suppress the incidence functions (and involution): we mostly write $\Gamma= (V,E)$ for a (directed) graph $\GG$ knowing that the incidence functions (and involution) are implicitly part of the specification.

A \emph{path} with start point $u$ and end point $v$ is a sequence of edges $e_1, \dots, e_n$ such that $\tar(e_i) = \sor(e_{i+1})$ for all $i$ and $\sor(e_1) = u$ and $\tar(e_n)= v$. A path $e_1, \dots, e_n$ in an undirected graph is called \emph{without backtracking} if $e_i \neq \ov e_{i+1}$ for all $i$.

 \medskip
 \noindent {\bf Schreier and Cayley graphs.}
 Let $G$ be a group and $P$ be a subgroup of $G$. 
The \ei{Schreier graph} $\GG = \GG(G,P,\Sig)$ of $G$ with respect to $P$ and set of monoid generators $\Sig$ of $G$ is defined as follows: The vertex set $V(\GG)$ is the set of 
right cosets $\lQuot{P}{G}= \set{Pg}{g \in G}$ and the edge set 
 $E(\GG)$ is the set $\lQuot{P}{G} \times \Sig$ with $\sor(Pg,a)=Pg$ and $\tar(Pg,a) = Pga$.
 For $\abs \Sig = d$ it is a $d$-regular directed graph. 
 If $\Sig$ is symmetric, then 
 $\GG(G,P,\Sig)$ is an undirected graph thanks to the involution 
 $\ov{(Pg,a)} = (Pg a,\ov a)$. 
 
 If $P$ is the trivial group, then $\GG(G,P,\Sig)$ is the \emph{Cayley graph} of $G$.
For an arbitrary subgroup $P$ there is a natural action of $P$ on the Cayley graph $\GG(G,\{1\},\Sig)$. The Schreier graph $\GG(G,P,\Sig)$ coincides with the quotient graph $\lquot{P}{\GG(G,\{1\},\Sig)}$.

Fix some starting point $Pg \in \lquot{P}{G}$. Now, every word $w = a_1 \cdots a_n \in \Sig^*$ defines a unique path $(Pg,a_1), (Pga_1,a_2),\dots, (Pga_1\cdots a_{n-1}, a_n)$ which starts in $Pg$ and follows the edges corresponding to the respective letters. 
 Conversely, every path starting at $Pg$ defines a unique word $w \in \Sig^*$.
 Moreover, some word represents an element in $P$ if and only if, when starting in $P \in \lquot{P}{G}$, the unique path described by $w$ also ends in $P$. 
 
 Likewise, there is a canonical bijection between reduced words and paths without backtracking starting at some fixed vertex.

 \medskip
\noindent {\bf Strongly generic algorithms.}
In algorithmic problems the inputs are taken from some 
specific domain $D$; in most cases the domain 
$D$ naturally comes as disjoint union $D = \bigcup\set{D^{(n)}}{n \in \N}$ such that each $D^{(n)}$ is finite. For example, $D^{(n)}$ is the set of words of length $n$ or the set of reduced words of length $n$ or the set of integers having a binary representation with $n$ bits, etc. 
A set $N\sse D$ is called \emph{strongly negligible} 
if
$$\frac{\abs {N\cap D^{(n)}}}{ \abs {D^{(n)}}} \in 2^{-\OO(n)}.$$
A set $L\sse D$ is called \emph{strongly generic} if its complement 
$D\sm L$ is strongly negligible.
Thus, as soon as $L$ is strongly generic,
if a random process chooses an element uniformly among 
all elements from $D^{(n)}$, then, for all practical purposes, we can ignore with increasing $n$ the event that it finds an element outside $L$. 
A problem $\cP$ is solved by a {\em strongly generic algorithm} $\cA$
if there is a strongly generic set $L$ such that the following three conditions hold: 
\begin{enumerate}
\item $\cA$ solves $\cP$ correctly on all inputs from $L$.
\item $\cA$ may refuse to give an answer or it might not terminate, but only on inputs outside $L$.
\item If $\cA$ gives an answer, then the answer \emph{must} be correct.
\end{enumerate}
The (time) complexity is the worst case behavior measured only on elements of $L$: for inputs in $L$ we count the maximal number of steps until the algorithm stops with the correct answer.

\section{Results}\label{sec:results}
Let $G$ be a finitely generated group and $\eta:\Sig^* \to G$ be a finite monoid presentation. In case that $\Sig$ is symmetric, 
let $\Del$ denote the subset of cyclically reduced words in $\Sig^*$.

\begin{thmy}\label{thm:one}\label{THMONE}
Let \begin{itemize}
\item $G =H\star_A K$ be an amalgamated product such that $[H:A]\geq 3$ and $[K:A] \geq 2$, or let
\item $G = \genr{H,\y}{\y a\y^{-1}=\phi(a) \text{ for } a \in A}$ be an HNN extension with $[H:A]\geq 2$ and $[H:\phi(A)] \geq 2$.
\end{itemize} 
Then the following holds:
\begin{enumerate}
 \item The set of words representing hyperbolic elements in $G$ is strongly generic in $\Sig^*$. \label{gia}
 \item If $\Sig$ is symmetric, then the set of cyclically reduced 
 words representing hyperbolic elements in $G$ is strongly generic in $\Del$, too. \label{gib} 
 \end{enumerate}
 \end{thmy}

The proof of \prref{thm:one} relies on the notion of 
an {\em amenable graph} given below in \prettyref{sec:repr} and the following two results 
 about amenable Schreier graphs.

 \begin{thmy}\label{thm:a}\label{THMA}
 Let $G =H\star_A K $ with $[H:A]\geq [K:A]\geq 2$ 
and $P \in \os {H,K}$ and let $\Sig$ be a finite symmetric set of generators. Then the Schreier graph $\GG(G,P,\Sig)$ is non-amenable if and only if $[H:A]\geq 3$.
\end{thmy}
\begin{thmy}\label{thm:h}\label{THMH}
Let $G= \genr{H,\stl}{\stl a\stl^{-1}=\phi(a) \text{ for } a \in A}$ be an HNN extension
and let $\Sig$ be a finite symmetric set of generators of $G$. Then the Schreier graph $\GG(G,H,\Sig)$ is non-amenable if and only if both $[H:A]\geq 2$ and $[H:\phi(A)]\geq 2$.
\end{thmy}

\begin{example}\label{ex:baumslag}
Let $\BS{p}{q}= \genr{a,t}{t a^pt^{-1}= a^q}$ the Baum\-slag-Solitar group with $1\leq p \leq q$. 
 Then the Schreier graph $\GG(\BS{p}{q}, \gen{a}, \os{a, \ov a, t, \ov t})$ is amenable if and only if $p= 1$. 
 
Actually, even the Cayley graph of $\BS{1}{q}$ is amenable \cite[Thm.\ 15.14]{woess2000}. 
 $\Diamond$\end{example}
\begin{example} Let $\BG= \genr{\BS12,b}{b ab^{-1}=t}$ the Baumslag group. We have shown in \cite{DiekertMW14} that the Schreier graph $\GG(\BG,\BS{1}{2}, \os{a, \ov a, b, \ov b})$ is non-amenable. This fact is now a special case of \prref{thm:h}. In \cite{DiekertMW14}, we also showed that the conjugacy problem is decidable in polynomial time for hyperbolic elements. Thus, together with \prref{thm:one}, we a obtain a strongly generic polynomial time algorithm for the conjugacy problem of $\BG$ with respect to arbitrary finite generating sets.
 $\Diamond$\end{example}

 \begin{example} Let $H_4= \genr{a_1,a_2,a_3,a_4}{a_{i+1} a_i a_{i+1}^{-1}=a_i^2 \text{ for } i\in \rquot{\Z}{4\Z}}$ be the Higman group. 
 Let
 \begin{align*}
  G_{\scriptscriptstyle 123} &= \genr{a_1,a_2,a_3}{a_{2} a_1 a_{2}^{-1}=a_1^2,\, a_{3} a_2 a_{3}^{-1}=a_2^2}, \\
  G_{\scriptscriptstyle 341} &= \genr{a_2,a_3,a_4}{a_{4} a_3 a_{4}^{-1}=a_3^2,\, a_{1} a_4 a_{1}^{-1}=a_4^2}.
 \end{align*}
Then we can write $H_4$ as amalgamated product
 $$H_4 = G_{\scriptscriptstyle 123}*_{\gen{a_1, a_3}} G_{\scriptscriptstyle 341}.$$
%
%
 By  \prref{thm:h}, the Schreier graph $\GG(H_4,G_{\scriptscriptstyle 123}, \set{a_i, \ov a_i}{i=1,\dots, 4})$ is non-amenable. We conjecture that with the same techniques as in  \cite{DiekertMW14} for the Baumslag group, the conjugacy problem can be solved in polynomial time for hyperbolic elements. By \prref{thm:one}, this would lead again to a strongly generic polynomial time algorithm for the conjugacy problem of $H_4$.
 $\Diamond$\end{example}

We postpone the proofs of \prref{thm:one}, \ref{thm:a} and \ref{thm:h} and present two more applications of these theorems. 
The first corollary shows that also the conjugacy problem of the $\Z^4$-by-free group of \cite{BogopolskiMV10} (with undecidable conjugacy problem) is strongly generically decidable in polynomial time.

\begin{corollary}\label{cor:freeamen}\label{CORFREE}
	If $G$ is a fundamental group of a finite graph of groups\footnote{For a definition we refer to \cite{serre80}.} with finitely generated free abelian vertex groups, in particular,
	\begin{itemize}
		\item  if $G =H\star_A K $ is an amalgamated product with $H$, $K$ finitely generated free abelian or
		\item  if $G= \genr{H,\stl}{\stl a\stl^{-1}=\phi(a) \text{ for } a \in A}$ is an HNN extension with $H$ finitely generated free abelian,
	\end{itemize}
	then the conjugacy problem of $G$ is decidable in polynomial time on a strongly generic set. 
\end{corollary}
	
Note that by \cite[Thm.~1.4.7]{Beeker11thesis} the conjugacy problem is decidable for words representing hyperbolic elements. Hence, in the case that the requirements of \prettyref{thm:one} are met, the result follows when we replace ``decidable in polynomial time'' by ``decidable''.

Be aware that it is not sufficient to require simply that $G$ is finitely generated instead of every vertex group finitely generated. Indeed, let $L\sse \Z$ be some undecidable set. Let $$G = \genr{\!a,s,t}{[s^kas^{-k}, s^\ell as^{-\ell}] = 1\; (k,\ell \in \Z),\, ts^ias^{-i}t^{-1} = s^{i+1}as^{-i-1}\: ( i\in L)\!}$$ (where $[\,\cdot\,,\,\cdot\,]$ denotes the commutator). Then $G$ is finitely generated and it is an HNN extension of $ H = \bigoplus_{i\in \Z} \Z = \genr{a_i \text{ for  } i \in \Z}{[a_k,a_\ell]=1 \text{ for } k,\ell \in \Z}$ with two stable letters:
$G = \genr{H,s,t}{sa_is^{-1} = a_{i+1},\, ta_jt^{-1} = a_{j+1} \text{ for } i\in \Z, j \in L}$.
Moreover, it has undecidable word problem since $ts^ias^{-i}t^{-1} = s^{i+1}as^{-i-1}$ if and only if $i\in L$.

\begin{proof}[Proof of \prref{cor:freeamen}]
The proof consists of two parts: first, we give a polynomial time algorithm to decide conjugacy for hyperbolic elements. After that, we show that in the cases where \prettyref{thm:one} is not applicable, the conjugacy problem can be also decided in polynomial time for elliptic elements. As \prettyref{thm:one} establishes that the hyperbolic elements form a strongly generic set, this proves the corollary.

 \medskip
 \noindent {\bf Algorithm for hyperbolic elements.} As we do not want to introduce all the notation for graphs of groups, we first present the algorithm for the case of an HNN extension. The general case follows exactly the same way~-- we provide the details at the end based on the notation and definitions from \cite{serre80}.

The algorithm relies on the result by Frumkin \cite{Frumkin77} and von zur Gathen and Sieveking \cite{gs78} that the existence of an integer solution of a system of linear equations can be checked in polynomial time, and if there is a solution, it can be computed in polynomial time, see also \cite[Cor.\ 5.3b]{Schrijver86}.

By fixing a basis of $H$, we can write $H=\Z^n$ for some $n \in\N$. As a subgroup of a free abelian group, $A$ itself is free abelian and we can fix a basis of $A$. Thus, we have $A=\Z^m$ as an abstract group.
That means, we can represent all elements of $H$ and $A$ as vectors of binary integers (with respect to the bases of $H$ and $A$). 
Moreover, we obtain two inclusions of $A$ into $H$ (via the identity and via $\phi$), which are represented by matrices $M_{1}, M_{-1}\in \Z^{n\times m}$. (Note that as these matrices need not to be square, in particular, they need not to be invertible.) Some vector $y \in \Z^n$ represents a group element in $A$ (resp.\ $\phi(A)$) if an only if there is some $x \in \Z^m$ with $M_1x=y$ (resp.\ $M_{-1}x=y$).
Hence, the subgroup membership problem of $A$ and $\phi(A)$ in $H$ can be decided in polynomial time by the result of \cite{Frumkin77, gs78}, and also the isomorphism $\phi$ is computable in polynomial time (by multiplying the solution $x$ by $M_{-1}$ (resp.\ $M_{1}$)). This allows us to effectively calculate Britton-reduced words and cyclically Britton-reduced words. Moreover, in every application of a Britton reduction the sizes of the binary representations of the occurring numbers increase only by a constant. Hence, the algorithm of performing Britton reductions runs in polynomial time.
Now, let 
\begin{align*}
 v&= t^{\eps_1}g_1 \cdots t^{\eps_k}g_k,
 &w&= t^{\del_1}h_1 \cdots t^{\del_\ell}h_\ell 
\end{align*}
be cyclically Britton-reduced with $g_i,h_i\in H = \Z^n$ represented as vectors of binary integers for all $i$. By Collins' Lemma (see \eg~\cite{LS01}), we know that, if $v$ and $w$ are conjugate, then $k=\ell$ and we may assume $\eps_i = \del_i$ for all $i$. 
Moreover, Collins' Lemma tells us that after a cyclic permutation, we may assume that $v\sim_G w$ if and only if there is some $a \in A \cup \phi(A)$ such that
$ava^{-1}=_Gw$. 
By the normal form theorem for HNN extensions (see \eg~\cite{LS01}), it follows that this is the case if and only if there are vectors $x_1,\dots , x_k$ with $x_i \in \Z^m = A$ such that
\begin{align*}
M_{\eps_i}x_i - M_{-\eps_{i+1}}x_{i+1} + g_i &= h_i & \text{for } 1 \leq i \leq k,
\end{align*}
where indices are calculated modulo $k$ and, as above, $M_{1}$ and $M_{-1}$ are the matrices representing the inclusions $\mathrm{id}:A\to H$ and $\phi: A\to H$. This is a system of linear integer equations. Its size is linear in the input size; hence, the existence of an integer solution is decidable in polynomial time by \cite{Frumkin77, gs78}.

In order to prove the result for arbitrary graphs of groups, note that for every vertex $v$ of the graph of groups, we have $G_v= \Z^{n_v}$ for some $n_v \in \N$, and for every edge $y$, we have $G_y= \Z^{m_y}$ for some $m_y \in \N$. The inclusions $G_y \to G_{\sor(y)}$ and $G_y \to G_{\tar(y)}$ are now represented by matrices $M_y \in \Z^{n_{\sor(y)}\times m_y}$ and  $M_{ \ov y} \in \Z^{n_{\tar(y)}\times m_y}$. Britton-reduced words can be computed in polynomial time as in the HNN case. For the generalization of Collins' Lemma to arbitrary graph of groups, see \cite{Horadam81}~-- here the conjugating element $a$ comes from some edge group. The analog of the normal form theorem for HNN extensions can be found in \cite[Sec.\ I.5.2]{serre80}, where the $x_i$ come from the respective edge groups. Again, we obtain a system of linear equations, which can be solved in polynomial time by \cite{Frumkin77, gs78}.

\medskip
\noindent {\bf Algorithm for elliptic elements.}
	It remains to show that in  the cases  $G =H\star_A K $ with  $[H:A] = [K:A] = 2$ or $G = \genr{H,\y}{\y a\y^{-1}=\phi(a) \text{ for } a \in A}$ with $A = H$, the conjugacy problem is decidable in polynomial time for elliptic elements. These are indeed all cases where \prref{thm:one} cannot be applied because a fundamental group of a reduced graph of groups with at least two edges can be written as HNN extension or amalgamated product where the edge subgroup has infinite index. Thus, \prref{thm:one} can be applied. As elliptic elements of the fundamental group, in particular, are elliptic with respect to this HNN or amalgamated product decomposition, they form a negligible set.

We have already seen how to produce cyclically Britton-reduced words in polynomial time. Thus, in both cases, by Collins' Lemma, we can assume that we are given two elements $g,h \in H\cup K$ (resp.\ $g,h \in H$) as vectors of binary integers with respect to the bases we have chosen.

In the amalgamated product case, by Collins' Lemma, we know that $g\sim_G h$ if and only if there is a sequence $g_1, \dots, g_m \in H \cup K $ with $g= g_1$ and $g_m= h$ and $g_i\sim_H g_{i+1}$ or $g_i\sim_K g_{i+1}$ for all $i$.
	As $H$ and $K$ are abelian, this means $g\sim_G h$ if and only if $g=_G h$.
		
Now, assume that  $G = \genr{H,\y}{\y a\y^{-1}=\phi(a) \text{ for } a \in A}$ with $A = H$. If also $\phi(A)=H$, then we can apply a result from \cite{CavalloK14} which states that in this case the conjugacy problem can be solved in polynomial time.
If $H \neq \phi(A)$, we can apply the same tool as in \cite{CavalloK14} to solve the conjugacy problem:
	
By Collins' Lemma for HNN extensions, we have $g\sim_G h$ if and only if  either $g\sim_H h$ or there is a sequence $g_1, \dots, g_m \in A \cup \phi(A)$ with $g_i\sim_H g_{i+1}$ or $g_i = \phi^{\pm 1}(g_{i+1})$ and $g\sim_H g_1$, $g_m\sim_H h$.
	
Again, because $H$ is abelian, this reduces to a check whether there is some $i \in \N$ with $g = \phi^i(h)$ or  $h = \phi^i(g)$. The latter question it also referred to as \emph{orbit problem}. As $\phi$ is represented by an integer matrix, we can apply a result by Kannan and Lipton \cite{KannanL80} which states that the orbit problem for rational matrices is decidable in polynomial time. This concludes the proof of \prref{cor:freeamen}.
\end{proof}

The next corollary is about the conjugacy problem in groups with more than one end\footnote{A group has more than one end if its Cayley graph can be split into two infinite connected components by removing some finite set of vertices. A more thorough definition can be found \eg\ in the survey \cite{Moeller95}.}. It shows that in a strongly generic setting the conjugacy problem is essentially as difficult as the word problem. Note that we do not require anything about the conjugacy problem at all!

 \begin{corollary}\label{cor:infendgeneric}
 Let $G$ be a finitely generated group with more than one end. 
 If the word problem of $G$ is decidable in polynomial time 
 (resp.~in time $\Oh(t(n))$ with $t(n)\geq n$), then there is a strongly generic algorithm 
 which solves the conjugacy problem of $G$ in polynomial time 
 (resp.~in time $\Oh(nt(n+ \Oh(1)))$)~-- in particular, the conjugacy problem of $G$ is decidable on a strongly generic set. 
\end{corollary}

\begin{proof}
 Due to Stallings' structure theorem \cite{Stallings71}, we have to consider 
two situations: either $G$ is an amalgamated product 
$G=H\star_A K$ with $H\neq A \neq K$ or 
$G$ is an HHN extension 
$G= \genr{H,\y}{\y a\y^{-1}=\phi(a) \text{ for } a \in A}$
where in both cases the edge group $A$ is finite.  

If 
\ifarXiv 
\else
we are in the case that
\fi $G=H\star_A K$ with $[H:A]= [{K:A}]= 2$ or $G=\genr{H,\y}{\y a\y^{-1}=\phi(a) \text{ for } a \in A}$) with $A=H$ (and thus $H=\phi(A)$), then the group $G$ is virtually cyclic \cite[Thm.\ IV.6.12]{DicksD89}. For 
virtually cyclic groups the conjugacy problem is very easy:
it can be solved in linear time. (In fact, all virtually free groups have conjugacy problem solvable in linear time, see \eg~\cite[Prop.\ 6.1]{ddm12}.)

Thus, we can assume that either we have $[H:A]\geq 3$ and $[K:A]\geq 2$ (if $G=H\star_A K$)
or $[H:A]\geq 2$ (if $G=\genr{H,\y}{\y a\y^{-1}=\phi(a) \text{ for } a \in A}$). Hence, by \prref{thm:one} over any finite alphabet the set of words representing hyperbolic elements is strongly generic.

We now describe an algorithm to solve the conjugacy problem which gives an answer as long as one of the input words is hyperbolic. We choose a finite symmetric generating set $\Sig$ of $G$ such that
$A\sse \Sig$ and
$\Sig \sse_G H\cup K$ (if $G=H\star_A K$) 
respectively $A, \phi(A)\sse \Sig$ and $\Sig \sse_G H\cup \oneset{t, \ov t}$ (if $G=\genr{H,\y}{\y a\y^{-1}=\phi(a) \text{ for } a \in A}$).

Let $u,v \in \Sig^*$ be input words with $\abs {uv}=n$. The question is whether $u$ and $v$ are conjugate. First, we apply Britton reductions to both words leading to Britton-reduced words $u'$ and $v'$.
 If the word problem of $G$ is decidable in time $t(n)$, we can perform Britton reductions for a word $w$ with $\abs w \leq n$ in time at most $\Oh(n t(n))$. Indeed, if we see, for example, a factor $tp\ov t$ where $p\in \Sig^*$ with $p =_Ga$ for some $a\in A$, then we can replace the factor $tp\ov t$ by $\phi(a)$, which is a word of length one. The other situations are similar.

Next, we apply a cyclic permutation to $u'$: in the HNN case, we write $u' =u'_1u'_2$ such that $\abs{\abs{u'_1}_t + \abs{u'_1}_{\ov t} - \abs{u'_2}_t- \abs{u'_2}_{\ov t}}\leq 1$ and $u'_1$ ends in $t$ or $\ov t$; then we apply Britton reductions to the word $u'_2u'_1$. Likewise, we proceed for $v'$. This leads to cyclically Britton-reduced words $u''$ and $v''$ such that $u\sim_G u''$ and $v\sim_G v''$. In the case of an amalgam, $u' =u'_1u'_2$ is factorized such that the number of alternations between the factors $H$ and $K$ in $u'_1$ and in $u'_2$ differ by at most one~-- then we proceed as for HNN extensions.

Collins' Lemma (see \eg~\cite{LS01}) tells us several things:
 $u$ is hyperbolic \IFF $u''$ does not belong to a vertex group~-- the same assertion holds for $v$ and $v''$.
By the very definition, hyperbolic elements are never conjugate to elliptic elements. Thus, if say $u$ is elliptic and $v$ is hyperbolic, then $u$ and $v$ are not conjugate. If both are elliptic, then the algorithm refuses the answer.
Thus, without restriction, henceforth $u$ and $v$ are both hyperbolic. 
In this case Collins' Lemma tells us that $u$ and $v$ are conjugate \IFF there is a cyclic permutation $u''_2u''_1$ of $u'' = u''_1u''_2$ and some $a \in A$ such that $au''_2u''_1 =_G v''a$. This can be checked with at most $n \abs{A}$ calls to the word problem (with inputs of length $n+2$).
We need time $\Oh(n t(n+\Oh(1)))$ to perform the entire algorithm. 
\end{proof}

\section{Random walks and amenability}\label{sec:amenability}
There is large body of literature on amenable groups, graphs, and 
 metric spaces as well as on different notions for random walks. In this section we review some of the known characterizations of amenability
for undirected $d$-regular graphs and the consequences for return probabilities of random walks in (directed) Schreier graphs. 
We consider $d$-regular directed graphs $\GG=(V,E)$, only. This means for each $v \in V$ there are exactly $d$ outgoing edges and 
 $d$ incoming edges. We allow self-loops and multiple edges. 
As a consequence, there are exactly $d^n$ different paths of length $n$ starting at a fixed vertex. Recall that undirected graphs are special cases of directed graphs.

\medskip
\noindent {\bf Random Walks.} 
The \ei{random walk}\index{random walk} on a directed graph 
is as follows: it starts at some vertex, chooses an outgoing edge uniformly at random and goes to the target vertex of this edge, then it chooses the next edge and so on.
With $p^{(n)}(u,v)$ we denote the probability that the random walk on $\GG$ ends after $n$ steps in $v$ when starting in $u$. Thus, 
\begin{align*}
 p^{(n)}(u,v) &= \frac{\text{number of paths from $u$ to $v$}}{d^n}.
\end{align*}

Similarly we can define the \emph{random walk without backtracking} on an undirected graph: it starts by choosing an edge starting at $u$ uniformly at random. For the following steps the inverse of the previous edge is excluded, \ie there are only $d-1$ possible choices each of which has equal probability. We obtain the following 
probability for $n \geq 1$: 
\begin{align*}
 q^{(n)}(u,v) &= \frac{\text{number of paths without backtracking from $u$ to $v$}}{d \cdot (d-1)^{n-1}}.
\end{align*}

We say that the random walk has \ei{exponentially decreasing return proba\-bi\-lity} if there are constants $c, \eps>0$ such that for all $n \in \N$ and $u,v\in V$ we have $p^{(n)}(u,v) \leq c2^{-\eps n}$.

 \medskip
 \noindent {\bf Spectral Radius.} 
 We can think of $V$ as a subset of $\R^V$ by 
 identifying a vertex 
 $u\in V$ with its characteristic function $x_u$. (Thus, $x_u(v) = 1$ if $u=v$ and $x_u(v) = 0$ otherwise.)
 We restrict our attention to the 
 Hilbert space of functions $x:V\to \R$ such that
 $\Abs{x} = \sqrt{\sum_{v\in V} x(u)^2} <\infty$. The inner product 
 is as usual $\scalp{x}{y} = \sum_{u\in V} x(u)y(v)$.
 
 The unit vectors $x_u$ span a dense vector space in the Hilbert space and
 $\GG$ defines a \emph{random walk operator} $R_\GG$ by
 letting 
 $$R_\GG(x_u) = \frac{1}{d} \sum\set{x_{\tau(e)}}{e\in E \wedge \iota(e) = u}.$$
 It is clear that $R_\GG^n(x_u)$ describes exactly the 
 \prob distribution of the random walk on $\GG$ of length $n$ starting in vertex $u$. 
 The \ei{spectral radius} $\rho(\GG)$ is defined as 
 $$\rho(\GG) = \sup\set{\scalp{x}{R_\GG(x)}}{\Abs{x}=1}.$$ 
 Let $R$ be any linear operator on the Hilbert space. Its {\em norm} $\Abs R$ is defined by
$$\Abs R = \sup\set{\sqrt{\abs{\scalp{R x}{R x }}}}{\Abs{x}=1 }.$$
 We have $\rho(\GG)\leq \Abs {R_{\GG}} \leq 1$. Indeed, $\rho(\GG)\leq \Abs {R_{\GG}}$ follows immediately by Cauchy-Schwarz. A slightly more complicated calculation shows that $ \Abs {R_{\GG}} \leq 1$, see \cite[Lem.\ I.3.12]{Soardi94}.

 \medskip
 \noindent {\bf Distance and $k$-th neighborhood.} 
 Let $\GG = (V,E)$ be an 
undirected graph. The \emph{distance} $d(u,v)$ between vertices $u$ and $v$ is defined by the length of a shortest path connecting $u$ and $v$~-- if there is such a path. Otherwise, we let $d(u,v) = \infty$.

For $k\in \N$ the $k$-th neighborhood $\cN^k(U)$ of a set of vertices $U\sse V$ is defined by 
$$\cN^k(U) = \set{v\in V}{\exists\, u \in U \, :\, d(u,v)\leq k}.$$
 
 \newcommand{\F}{U}

\medskip

\noindent {\bf Amenability and return probabilities.}
 The following proposition is well-known. For Cayley graphs the equivalence of \ref{doubling}--\ref{randomw} goes back to Kesten \cite{Kesten59_2,Kesten59_1}. The generalization to arbitrary graphs of bounded degree appeared in \cite{Gerl88}. Condition \ref{gromov} is due to Gromov (see \cite[Condition $0.5.C_{1}''$]{Gromov93}). The result \ref{randomwwbt} about random walks without backtracking was shown independently by Cohen \cite{Cohen82} and Grigorchuk \cite{Grigorchuk77} for Cayley graphs of finitely generated groups. They used the notion of \emph{cogrowth}. The generalization to arbitrary $d$-regular graphs was proven by Northshield \cite{Northshield92}. 
Proofs of the equivalence of conditions \ref{gromov}--\ref{randomw} can also be found in Thm.\ 32 and Thm.\ 51 of \cite{CeccheriniSilbersteinGH98} and Thm.\ 4.27 of \cite{Soardi94}.
 \begin{proposition}\label{prop:amenability}
 Let $\GG= (V,E)$ be a $d$-regu\-lar undirected graph. Then the following statements are equivalent:
 \begin{enumerate}
 \item $\GG$ satisfies the \emph{Gromov condition}: there exists a map $f:V \to V$ such that \label{gromov} 
 \begin{itemize}\vspace*{-1mm}
 	\item $\sup_{v\in V}d(f(v),v) < \infty$ and
 	\item $\abs{f^{-1}(v)}\geq 2$ for all $v\in V$.
 \end{itemize} 
 
 \item $\GG$ satisfies the \emph{doubling condition}: there exists some $k\in \N$ such that for every finite $\F\sse V$ we have\label{doubling}\vspace*{-1mm}
 $$\abs{\cN^k( \F)}\geq 2 \abs{\F}.$$\vspace*{-5mm}
 \item The spectral radius is less than one: $\rho(\GG)<1$.\label{specr}
 \item The random walk on $\GG$ has exponentially decreasing return proba\-bi\-lity.\label{randomw}
 \item The random walk on $\GG$ without backtracking has exponentially decreasing return proba\-bi\-lity.\label{randomwwbt}
 \end{enumerate} 
 \end{proposition}

 \begin{definition}\label{def:amgr}
A graph is called \emph{non-amenable} if it satisfies one of the equivalent conditions of \prettyref{prop:amenability}. Otherwise, it is called \emph{amenable}.
 
 A finitely generated group is called \emph{non-amenable} (resp.\ \emph{amenable}) if it has a non-amenable (resp.\ amenable) Cayley graph with respect to some finite symmetric generating set. 
\end{definition}
Sometimes non-amenable graph are also called \ei{infinite expanders}  because in finite graph theory an undirected $d$-regular graph is an expander \IFF the second largest eigenvalue of its random walk matrix is strictly less than $1$. Condition~\ref{specr} plays an analogue role for infinite graphs.

 \begin{remark}\label{rem:generic}
 Conditions \ref{randomw} and \ref{randomwwbt} of \prref{prop:amenability} establish the connection between amenability of Schreier graphs and strongly generic sets: We noticed already that there is a canonical bijection between $\Sig^*$ and the set of paths starting at the origin $P$ of the Schreier graph $\GG(G,P,\Sig)$. Moreover, some word represents a group element in $P$ if and only if the respective path ends in $P$. Thus, the random walk on $\GG(G,P,\Sig)$ has exponentially decreasing return probability if and only if the set of words representing group elements in $P$ is strongly negligible.
 
 The same observation holds for the random walk without backtracking and reduced words representing group elements in $P$.
 $\Diamond$
 \end{remark}

 The notion of amenability is originally for groups, where it was first defined via invariant means.
Using the F\o lner condition \cite{Folner55} (a slight modification of the doubling condition \ref{doubling}), it can be seen that this definition coincides with \prref{def:amgr}.
Moreover, it is well-known that amenability is a property of the group and does not depend on its symmetric generating set. This can be generalized to Schreier graphs: 
 \begin{corollary}\label{cor:indgen}
 Let $G$ be a finitely generated group and $P$ be a subgroup of $G$. Let $\Sig_{1}$ and $\Sig_{2}$ be two finite symmetric generating sets of $G$. Then, $\GG(G,P,\Sig_{1})$ is amenable if and only if $\GG(G,P,\Sig_{2})$ is amenable. 
 \end{corollary}
 \begin{proof}
 This is trivial consequence of \prref{prop:amenability} because the Gromov condition \ref{gromov} is invariant under the change of finite generating sets. 
 \end{proof}
 
Clearly, \prref{cor:indgen} has been known before. For example it can be found in \cite[Prop 6.2]{KMSS1} which in turn refers to \cite[Prop.\ 38]{CeccheriniSilbersteinGH98}. 
Unfortunately, \cite[Prop.\ 38]{CeccheriniSilbersteinGH98} is not correct as it is stated there. We give a counterexample in \prettyref{ex:fehler}. Still, the mistake did not affect the correctness of the proof in \cite{KMSS1} as only a special case of the statement is used. 

\begin{remark}\label{rem:quotient}
      If a group $G$ acts freely on an amenable graph $\GG$ (\ie without fixed points), then the quotient graph $\lquot{G}{\GG}$ is amenable, too. This is because every path in $\lquot{G}{\GG}$ has a unique lifting in $\GG$ from a fixed start point. 
      In particular, if the lifted path returns to the starting point, then the path in  $\lquot{G}{\GG}$ does so, too.
      Thus, if the random walk on $\GG$ does not have exponentially decreasing return probability, neither does the random walk on $\lquot{G}{\GG}$. Hence, the negation of condition \ref{randomw} in \prref{prop:amenability} transfers from $\GG$ to  $\lquot{G}{\GG}$.
      
      In particular, if $G$ is an amenable group, then the Schreier graph $\GG(G,P,\Sig)$ is amenable with respect to every subgroup $P\leq G$ and finite symmetric generating set $\Sig$.
      $\Diamond$
\end{remark}

 It is a classical fact that groups of subexponential growth are amenable, see \eg\ \cite[Thm.\ 66]{CeccheriniSilbersteinGH98}. 
 As a direct consequence, all virtually nilpotent and, in particular, all abelian groups are amenable. On the other hand, non-abelian free groups are non-amenable. This can be verified for example using the Gromov condition \ref{gromov}: The function $f$ can be defined by deleting the last letter of a reduced word and letting $f(1)$ arbitrary, see \prref{fig:freeCayley}.
 
 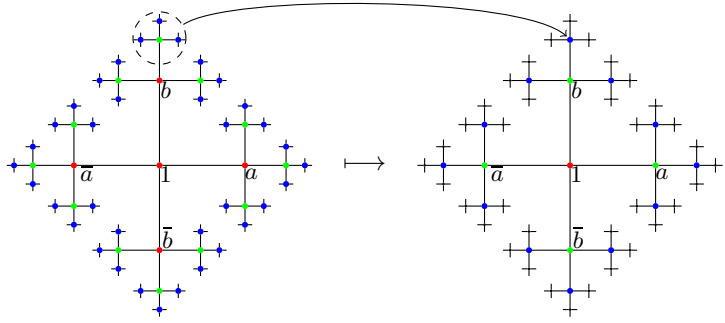
\begin{figure}
 \begin{center}
 {
 \tikzstyle{level 1}=[sibling angle=90]
 \tikzstyle{level 2}=[sibling angle=90]
 \tikzstyle{level 3}=[sibling angle=90]
 \tikzstyle{level 4}=[sibling angle=90]
 \tikzstyle{level 5}=[sibling angle=90]
 \tikzstyle{level 6}=[sibling angle=90]
 
 \tikzstyle{every node}=[ fill,inner sep=0pt]
 
 \begin{tikzpicture}[ scale=.45,grow cyclic,level distance=25mm, inner sep = 0pt, outer sep=0pt ]
 grow=left;
  \node[fill=none] (arrow) {$\longmapsto$};
  \begin{scope}[node distance = 60mm, left of = arrow]
  \draw[dashed] (0,3.75) circle (22pt);
  \draw[->]  (.7,4.15) ..controls +(1.3,.9) and +(-1.3,1.2).. (11.9,3.8);
 \begin{scope}[ rotate=45]
 {\node[circle, color=red,inner sep=.8pt] (1) {} child [-] foreach \A in {br,a,b,ar}
 	{ node[circle, color=red,inner sep=.8pt](\A){} child [-,level distance=12mm ] foreach \B in {red,green,blue}
 		{ node[circle, color=green,inner sep=.8pt] {} child [-,level distance=5.5mm ] foreach \C in {black,gray,white}
 			{ node[circle, color=blue,inner sep=.8pt] {} child [-,level distance=2.5mm ] foreach \C in {black,gray,white}
 				{ node[fill=none,inner sep=0pt,circle,outer sep=0pt] {} 
 				}
 			}
 		}
 	}
 	;}
 \small
 \path (1) -- ++(260:0.32) node[fill=none] {$1$};
 \path (a) -- ++(260:0.32) node[fill=none] {$a$};
 \path (b) -- ++(260:0.32) node[fill=none] {$b$};
 \path (ar) -- ++(283:0.44) node[fill=none] {$\ov a$};
 \path (br) -- ++(13:0.44) node[fill=none] {$\ov b$};
 
 \end{scope}
   \end{scope}

 \begin{scope}[node distance = 60mm,right of = arrow] 
 \begin{scope}[ rotate=45]

 {\node[circle, color=red,inner sep=.8pt] (1) {} child [-] foreach \A in {br,a,b,ar}
 	{ node[circle, color=green,inner sep=.8pt](\A){} child [-,level distance=12mm ] foreach \B in {red,green,blue}
 		{ node[circle, color=blue,inner sep=.8pt] {} child [-,level distance=5.5mm ] foreach \C in {black,gray,white}
 			{ node[inner sep=0pt,circle,outer sep=0pt] {} child [-,level distance=2.5mm ] foreach \C in {black,gray,white}
 				{ node[fill=none,inner sep=0pt,circle,outer sep=0pt] {} 
 				}
 			}
 		}
 	}
 	;}
 \small
 \path (1) -- ++(260:0.32) node[fill=none] {$1$};
 \path (a) -- ++(260:0.32) node[fill=none] {$a$};
 \path (b) -- ++(260:0.32) node[fill=none] {$b$};
 \path (ar) -- ++(283:0.44) node[fill=none] {$\ov a$};
 \path (br) -- ++(13:0.44) node[fill=none] {$\ov b$};
  \end{scope}
    \end{scope}
 \end{tikzpicture}
}

 	 \caption{The Cayley graph of $F({a,b})$. The function $f$ maps every vertex to its parent (one step towards the center).}\label{fig:freeCayley} \vspace{-3mm}
 	\end{center}
 \end{figure}
 
 Condition~\ref{randomw} of \prref{prop:amenability} can be defined for all directed Schreier graphs (not only undirected ones).
 However, \ref{randomw} depends on the chosen set of 
 monoid generators. Moreover, \prref{ex:directed} shows that, in general, for directed graphs conditions \ref{specr} and \ref{randomw} in \prref{prop:amenability} are not equivalent. 
 
 \begin{example}\label{ex:directed}
 Let $G= \Z$. Compare the following two directed 
 and different Cayley graphs with respect to the letters 
 $a= -1$, $\ov a = 1$, and $b=2$.
 
The Cayley graph of $G$ with 
respect to the symmetric generating set 
$\Sig= \os {a,\ov a}$ is amenable; and the random walk has a return \prob 
in $\OO(\frac{1}{\sqrt n})$ which is not exponentially decreasing.
For $\Sig'= \os {a,b}$ we obtain directed $2$-regular Cayley graph. The random walk on $\GG(G,\oneset{1},\Sig')$ has exponentially decreasing return probability. More precisely, the return \prob is at most $\binom{n}{\ceil{n/3}} 2^{-n}$. On the other hand, we have $\rho(\GG)=1$. To see this define 
 $x_n(v)=\frac{1}{\sqrt n}$ if $v \in \oneset{1, \dots, n}$ and $x_n(v)=0$ otherwise for $n\geq 1$. Then we have $\Abs{x_{n}} = 1$ and 
 $\lim\limits_{n\to \infty} \scalp{x_n}{R_\GG x_n} = 1$.
$\Diamond$ \end{example}

For a directed graph $\GG=(V,E)$, we can construct an undirected graph $\GG'=(V,E \cup \ov E)$, where $\ov E$ is a disjoint copy of $E$, and $\sor(\ov e)=\tar (e), \tar(\ov e)=\sor (e)$ for $\ov e \in \ov E$. If $\GG$ is $d$-regular, then $\GG'$ is $2d$-regular. We call $\GG'$ the \ei{undirected version} of $\GG$.
 Concerning random walks on directed graphs, we obtain \prettyref{lem:directed}, which is a special case of \cite[Thm.\ 10.6]{woess2000}. As the statement in \cite{woess2000} is slightly different, we present a proof.

 \begin{lemma}\label{lem:directed}
Let $\GG=(V,E)$ be a $d$-regular directed graph. 
If $\rho(\GG) <1$ (the spectral radius is less than $1$), then the undirected version $\GG'=(V,E \cup \ov E)$ of $\GG$ is non-amenable. 

Moreover, if $\GG'$ is non-amenable, then the random walk on $\GG$ has exponentially decreasing return probability.
 \end{lemma}
 
\begin{proof}
 Let $R$ be the random walk operator with $\rho(\GG) = \rho(R) < 1$. 
 Recall that we have $\rho(R) \leq \Abs R \leq 1$ for all random walk operators defined by directed regular graphs. Hence, 
for $P = 1/2(R + I)$ we obtain $\rho(P) < 1$. 
 Define $Q = 1/2(P + \oT P)$; then we have $\rho(Q) < 1$, too.
 Moreover, up to self-loops, $Q$ is the random walk operator of the undirected graph $\GG'$. 
 Self-loops do not effect the amenability. Hence, $\GG'$ is non-amenable. 
 
 Now, let $\GG'$ be non-amenable. Hence, $\rho(Q)<1$. As $P$ has self-loops on every vertex, $\oT P \cdot P$ is the random walk operator of a graph obtained from $Q$ by adding more edges. Hence, for example by the doubling condition,
this graph is non-amenable, too. This means $\rho(\oT P\cdot P) <1$. Now 
 $\scalp{x}{\oT P Px} = \scalp{Px}{Px}$ 
 shows $\Abs P < 1$. 
 Next we see that the return probability of the random walk on the graph defined by $P$ is exponentially decreasing because $p^{(n)}(u,v) = (P^nx_u)(v) \leq\Abs {P^n} \leq {\Abs P}^n \in 2^{-\OO(n)}$.
 Again, we can ignore the difference between between $R$ and $P$ because $P$ simply contains more self-loops. Hence, the random walk on $\GG$ has exponentially decreasing return probability.
\end{proof}

 Note that the converse of \prettyref{lem:directed} is not true, in general, (even not under the stronger assumption of $\GG$ being strongly connected) as we have seen in \prettyref{ex:directed}.

Next, we combine \prettyref{cor:indgen} (independence of the 
symmetric generating set) with \prettyref{lem:directed}: 
\begin{proposition}\label{prop:indgen}
 Let $G$ be a finitely generated group and $P$ be a subgroup of $G$. Let $\Sig$ be a finite symmetric generating set of $G$. If $\GG(G,P,\Sig)$ is non-amenable and $\Sig'$ is a finite monoid generating set of $G$, then the random walk on $\GG(G,P,\Sig')$ has exponentially decreasing return probability.
 \end{proposition}

 \subsection*{A counterexample for Propopsition 38 of \cite{CeccheriniSilbersteinGH98}}
 The fact that amenability of Schreier graphs does not depend on the 
 symmetric generating set can be derived easily from \cite[Prop.\ 38]{CeccheriniSilbersteinGH98}, see \cite[Prop 6.2]{KMSS1}. However, the proposition does not hold in full generality as stated in \cite{CeccheriniSilbersteinGH98}, see 
 \prref{ex:fehler} below. (It still holds for graphs of bounded degree and the proof in \cite[Prop 6.2]{KMSS1} is therefore not affected.)
 We need the notion of quasi-isometry:
 \begin{definition}
 	A \ei{quasi-isometry} between graphs 
 	$\GG= (V,E)$ and $\GG'= (V',E')$ is a function $f: V \to V'$ such that there is some constant $C >0 $ satisfying the following conditions:
 	\begin{itemize}
 		\item $\displaystyle{\vphantom{\frac{1}{\ov C}}\frac{1}{C} \cdot d_\GG(x,y) - C \leq d_{\GG'}(f(x),f(y)) \leq C\cdot d_\GG(x,y) + C}$ for all $x, y \in V$,
 		\item for every $y \in V'$ there exists some $x \in V$ such that $d_{\GG'}(y,f(x)) \leq C$.
 	\end{itemize}
 \end{definition}
 
 \cite[Prop.\ 38]{CeccheriniSilbersteinGH98} asserts that condition \ref{gromov} of \prref{prop:amenability} is a quasi-isometry
 invariant for discrete metric spaces; thus, in particular, for arbitrary locally finite undirected graphs.
 However, this is not true, in general. 
 There are locally finite graphs satisfying \prref{prop:amenability}~\ref{gromov} and which are quasi-isometric to amenable graphs:

 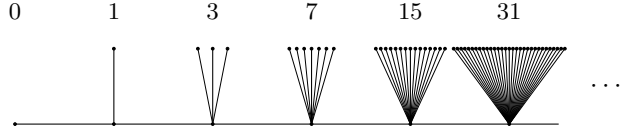
\begin{figure}[ht]
 	\begin{center}
 		\vspace{1.5mm}
 		\begin{tikzpicture}[ xscale=1.3,outer sep=0pt, inner sep = 0.15pt]
 		\node[draw, circle, fill] at (0,0){};
 		\node at (0,1.5){\small$0$};
 		\foreach \x in {0,1,2,3,4}
 		{
 			\node[draw,circle, fill] at (\x+1,0){};
 			\draw (\x,0) -- (\x+1,0);
 			
 			\draw (\x+1,0) -- (\x+1,1);
 			\node[draw,circle, fill] at (\x+1,1){};
 			
 			\node at (\x+1,1.5){\small$\pgfmathparse{int(2^(\x+1)-1)}\pgfmathresult$};
 			\pgfmathparse{(2^(\x+1)-2)/2}
 			\ifthenelse{\equal{\x}{0}}{}
 			{
 				\foreach \y in {1,...,\pgfmathresult}
 				{
 					{
 						\draw[line width=0.2pt] (\x+1,0) -- (\x+1 + 0.15*\y/\x,1);
 						\node[draw,circle, fill] at (\x+1 + 0.15*\y/\x,1){};
 						\node[draw,circle, fill] at (\x+1 - 0.15*\y/\x,1){};
 						\draw[line width=0.2pt] (\x+1,0) -- (\x+1 - 0.15*\y/\x,1);
 					}
 				};
 			}
 		}
 		\draw (5,0) -- (5.5,0);
 		\node at (6,0.5){$\cdots$};
 		\end{tikzpicture}
 		\caption{Counterexample for \cite[Prop.\ 38]{CeccheriniSilbersteinGH98}.}\label{fig:counterex} \vspace{-3mm}
 	\end{center}
 \end{figure}
 \begin{example}\label{ex:fehler}
 	Consider an infinite sequence of rooted trees $T_{k}$ each with $2^k$ nodes of which 
 	$2^k-1$ are leaves for $k = 0, 1,2, \ldots$. Connect the root of $T_{k}$ with the root of 
 	$T_{k+1}$, see \prettyref{fig:counterex} for the resulting undirected graph $\GG$ (the numbers displayed there are the numbers of leaves of the trees $T_k$). Since there is a two-to-one mapping of $T_{k+1}$
 	onto $T_{k}$, the graph satisfies the Gromov condition from \prref{prop:amenability}~\ref{gromov}. On the other hand, mapping 
 	all leaves to their roots we obtain a quasi-isometry onto 
 	a graph which is amenable.
 	$\Diamond$ 
 \end{example}

 \section{Proofs for Theorems~\ref{THMONE}, \ref{THMA}, and \ref{THMH}}\label{sec:repr}
\subsection{Amenability of Schreier graphs for amalgams}\label{sec:amas}
The aim of this section is to prove \prettyref{thm:a}. 
Thus, we consider an amalgamated product $G=H \star_A K$ with vertex groups 
$H$ and $K$ and edge group $A$. 
We start by choosing transversals $C\subseteq H$
and $D\subseteq K$ for cosets of $A$ in $H$ and in $K$ with $1 \in C \cap D$ such that there are unique decompositions
$H = AC$ and $K= AD$.
The following lemma is easy to see, c.f.~\cite[Sec.\ 7.4]{ddm10}). 
\begin{lemma}\label{lem:amalgNF}
Every group element $g \in G$ can be uniquely written as
\vspace{-2mm}$$g =_G x_0 \cdots x_k$$
for some $k\in \N$, $x_0 \in H \cup K$ such that for all $1 \leq i \leq k$ we have\vspace{-1.8mm}
\begin{align*}
x_i &\in C\cup D \setminus \oneset{1};\\
x_{i-1} &\in H \iff x_{i} \in K.
\end{align*}
\end{lemma}

 \begin{remark}\label{rem:amalgamen}
 Let $G =H\star_A K$ with $[H:A]=[K:A]=2$ and $P\in \oneset{H,K}$, $\Sig \sse_G A \cup \os{h,\ov h, k,\ov k }$ a finite symmetric generating set with $h\in H$ and $k\in K$.  Then we have 
 $$ p^{(2n)}(P,P)\in \Theta(n^{-1/2}).$$ 

 As $A$ is normal in both $H$ and $K$, we can rewrite any word $w\in \Sig^*$ to a unique normal form $x_0 \cdots x_\ell$ 
for some $ \ell \in \N$ such that $x_0 \in H \cup K$, $x_i \in \oneset{h,k}$ for $i>0$ and letters $h$ and $k$ always alternate. 
	
	Let $P\in \os {H,K}$ one of the vertex groups. Up to self-loops, the Schreier graph $\GG(G,A,\Sig)$ is isomorphic to the 
Cayley graph $\GG_{\Z}$ of $\Z$ with respect to the natural symmetric generating set $\os{\pm 1}$. As $A$ is normal in $P$, there is an action of $P$ (more precisely, of $P/A$) on  $\GG(G,A,\Sig)$. Now, the Schreier graph  $\GG(G,P,\Sig)$ is obtained as quotient $\lquot{P}\GG(G,A,\Sig)$. This mean the edge $(A, h)$ (resp.\ $(A, k)$) is ``folded''~-- \ie its endpoints are identified, see \prettyref{fig:amalgSchreier}.

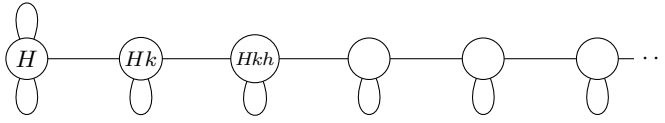
\begin{figure}
 \begin{center}
 	\vspace{-3mm} 	
 	\begin{tikzpicture}[node distance = 15mm, xscale=1.6,outer sep=0pt, inner sep = 0pt]
 	\small
 	\node[minimum size = 5.5mm, draw,circle,  inner sep =1pt] (a){$H$};
 	\node[minimum size = 5.5mm,draw,circle,  inner sep =1pt, right of = a] (b){\footnotesize$Hk$};
 	\node[minimum size = 5.5mm,draw,circle,  inner sep =1pt, right of = b] (c){\scriptsize $H\:\!\!k\;\!\!h$};
 	\node[minimum size = 5.5mm,draw,circle,  inner sep =1pt, right of = c] (d){};
 	\node[minimum size = 5.5mm,draw,circle,  inner sep =1pt, right of = d] (e){};
 	\node[minimum size = 5.5mm,draw,circle,  inner sep =1pt, right of = e] (f){};
  	
 	\draw (a) edge (b);
 	\draw (b) edge (c);
 	\draw (c) edge (d);
 	\draw (d) edge (e);
 	\draw (e) edge (f);
 	\draw (f) edge +(.3,0);
 		
 	\draw  (a) ..controls +(-.2,-.9) and +(.2,-.9).. (a);%
 	\draw  (a) ..controls +(-.2,.9) and +(.2,.9).. (a);%
 	\draw (b) ..controls +(-.2,-.9) and +(.2,-.9).. (b);
 	\draw (c) ..controls +(-.2,-.9) and +(.2,-.9).. (c);
 	\draw (d) ..controls +(-.2,-.9) and +(.2,-.9).. (d);
 	\draw (e) ..controls +(-.2,-.9) and +(.2,-.9).. (e);
 	\draw (f) ..controls +(-.2,-.9) and +(.2,-.9).. (f);
	
 	\node[circle,  inner sep =2pt,node distance = 8mm, right of = f]{$\cdots$};
 	
 	\end{tikzpicture}
	 \caption{The Schreier graph $\GG(H\star_A K, H, \Sig)$ for  $[H:A]=[K:A]=2$.}\label{fig:amalgSchreier} \vspace{-3mm}
	\end{center}
\end{figure}
The return \prob in $\GG_{\Z}$ (without self-loops) is $\binom{n}{n/2}2^{-n}\in \Theta(n^{-1/2})$ for $n$ even. Adding self-loops changes this only by a multiplicative constant. Folding the edge does not decrease the return probability (compare to \prref{rem:quotient}). Thus, we have $p^{(2n)}(P,P) \in \Theta(n^{-1/2})$
whether or not there are self-loops.
 $\Diamond$ \end{remark}

\begin{proof}[Proof of \prettyref{thm:a}]
For the only-if direction we assume $[H:A]=[K:A] = 2$. Then the Schreier graph $\GG(G,P,\Sig)$ is amenable due to \prettyref{rem:amalgamen} and \prettyref{prop:amenability}.
 
 For the other direction let $[H:A]\geq 3$. We show condition \ref{gromov} in \prettyref{prop:amenability}. The idea is the same as for the free group, see \prref{fig:freeCayley}. We define a function $f: \lquot{P}{G}\to \lquot{P}{G}$ as follows: We fix $c,c'\in C $ with  $c\neq 1 \neq c'$ and $c\neq c'$ and some $1\neq d \in D$. 
 For a normal form $w$ (according to \prettyref{lem:amalgNF}) with $w=vcd$ or $w=vc'd$ for some word $v$, we set $f(Pw) = P v$. Likewise, for a normal form $w$ with $w=vdc$ or $w=vdc'$, we set $f(Pw) = P v$. 
 Otherwise, we set $f(Pw) = P w$. Due to \prettyref{lem:amalgNF}, the function $f$ is well-defined.
Let $k\in \N$ be some number which is large enough such that 
 $c$, $c'$, and $d$ can be written with at most $k$ letters from $\Sig$. Then we obtain $\sup \set{d(f(Pw),P w)}{Pw \in \lquot{P}{G}}\leq 2k<\infty$. 
 For every normal form $w$, either $wcd$ and $wc'd$ or $wd c$ and $wd c'$ are normal forms. Hence, we have $\abs{f^{-1}(P w)}\geq 2 $ for all $w\in G$.
\end{proof}

\subsection{Amenability of Schreier graphs for HNN extensions}\label{sec:HNNs}
The aim of this section is to prove \prettyref{thm:h}. 
Thus,
let $$G = \genr{H,\stl}{\stl a\stl^{-1}=\phi(a) \text{ for } a \in A}$$ be an HNN extension of $H$ with an isomorphism $\phi :A\to B$ of subgroups $A$, $B$ of $H$. 
We can prove an analog result as for amalgamated products by the same strategy: first, we define certain normal forms, second, we show that 
the Schreier graph is amenable if $H=A$, and third, we show, using again the Gromov condition, that it is non-amenable for $A\neq H\neq \phi(A)$.

To begin with, we choose transversals for cosets of $A$ and $B = \phi(A)$ in $H$;
that is $C,D\subseteq H$ with $1 \in C \cap D$ such that there are unique decompositions
$H = AC$ and $H= BD$.
The next lemma is easy to see again, c.f.~\cite[Sec.\ 7.3]{ddm10}. 
It is the analogue of \prettyref{lem:amalgNF}. In order so simplify notation, we write $\oi t$ for the letter $\ov t$.
\begin{lemma}\label{lem:hnnNF}
Every group element $g \in G$ can be uniquely factorized (as a word over $H \cup \oneset{\stl,\ov\stl}$) as
$$g =_G x_0t^{\eps_1}x_1 \cdots t^{\eps_k}x_k$$
such that $k\in \N$, $x_0 \in H $, and
for all $1 \leq i \leq k$ we have:
\begin{itemize}
\setlength{\itemsep}{3.5pt}
\item $x_i \in C\cup D$ and $\eps_i \in \oneset{\pm 1}$;
 \item if $\eps_i=1$, then $x_i \in C$;\ \ if $\eps_i=-1$, then $x_i \in D$;
 \item if $\eps_i=-\eps_{i+1}$ and $i<k$, then $x_i \neq 1$.
\end{itemize}
\end{lemma}

 \begin{remark}\label{rem:hnnamen}
 Let $G= \genr{H,\stl}{\stl a\stl^{-1}=\phi(a), a \in A}$ with $A=H$ (and $B=H$ or $B\neq H$). Let $\Sig \sse_G H \cup \oneset{t, \ov t}$ be a finite generating set of $G$. 
 Then we have
 $$ p^{(2n)}(H,H)\in \OO(n^{-3/2}).$$ 
 
Let $0 \leq m \leq n$. Consider a word $w\in \Sig^*$ of length $2n$ with $m=\abs{w}_t = \abs{w}_{\ov t}$ such that every prefix $w'$ of $w$ satisfies $\abs{w'}_t \geq \abs{w'}_{\ov t}$.
Then we have $w \in H$. This is because always some Britton reduction can be applied (replacing $\stl a\ov t$ by $\phi(a)$), and hence the word can be reduced to some word in $H$. Restricted to $\os{t,\ov t}$, these words are exactly the Dyck words of length $2m$. The number of Dyck words of length $2m$ is the $m$-th Catalan number $\frac{1}{m+1}\binom{2m}{m}$, see \eg\ \cite{gkp94}.
We have $\frac{1}{m+1}\binom{2m}{m} \in \Theta(m^{-3/2}\cdot 4^m)$. 
As $2m\leq 2n$ we obtain the desired result.
$\Diamond$ \end{remark}

\begin{example}\label{ex:SchreierBS}
  	
  Consider the Schreier graph $\GG(\BS{1}{2}, \gen{a}, \oneset{a,\ov a, t,\ov t})$ of the Baumslag-Solitar group $\BS12$, see \prref{fig:SchreierBS}.
The branch to the right are those vertices $\gen{a}\!w$ with $\abs{w}_t \geq \abs{w}_{\ov t}$. As long as the random walk does not leave this branch, we are in the situation as in \prref{rem:amalgamen}. Thus, the return probability remains polynomially high.
\begin{figure}
	\begin{center}
		\vspace{-3mm} 	
		\begin{tikzpicture}[node distance = 15mm, xscale=1.6,outer sep=0pt, inner sep = 0pt]
		\small
		\node[state,accepting,minimum size = 7mm, draw,circle,  inner sep =1pt] (a){$\gen{a}$};
		\node[minimum size = 7mm,draw,circle,  inner sep =1pt, right of = a] (b){\footnotesize$\gen{a}\!t$};
		\node[minimum size = 7mm,draw,circle,  inner sep =1pt, right of = b] (c){\tiny $\gen{a}\!\!\!\:tt$};
		\node[minimum size = 7mm,draw,circle,  inner sep =1pt, left of = a] (d){ \footnotesize $\gen{a}\!\ov t$};
		\node[minimum size = 7mm,draw,circle,  inner sep =1pt,node distance = 23mm, left of = d] (e){\tiny $\gen{a}\!\!\!\:\ov t \ov t$};
		\node[minimum size = 5mm,draw,circle,  inner sep =1pt,node distance = 12mm, left of = e] (f){};
		
		\node[minimum size = 5mm,draw,circle,  inner sep =1pt,node distance = 25mm, above of = d] (da){ };
		\node[minimum size = 3mm,draw,circle,  inner sep =1pt,node distance = 12mm, above of = e] (ea){};
		\node[minimum size = 2mm,draw,circle,  inner sep =1pt,node distance = 6mm, above of = f] (fa){};
		
		\node[minimum size = 5mm,draw,circle,  inner sep =1pt,node distance = 12mm, right of = da] (datz){ };
		\node[minimum size = 5mm,draw,circle,  inner sep =1pt,node distance = 12mm, right of = datz] (dat){ };
		\node[minimum size = 3mm,draw,circle,  inner sep =1pt,node distance = 23mm, left of = da] (dad){ };
		\node[minimum size = 3mm,draw,circle,  inner sep =1pt,node distance = 10mm, above of = dad] (dada){ };
		\node[minimum size = 2mm,draw,circle,  inner sep =1pt,node distance = 12mm, left of = dad] (dadd){ };
		\node[minimum size = 2mm,draw,circle,  inner sep =1pt,node distance = 5mm, above of = dadd] (dadda){ };

				\node[minimum size = 3mm,draw,circle,  inner sep =1pt,node distance = 7mm, right of = dada] (dadatz){ };
				\node[minimum size = 3mm,draw,circle,  inner sep =1pt,node distance = 7mm, right of = ea] (eatz){ };
				
		\node[minimum size = 3mm,draw,circle,  inner sep =1pt,node distance = 7mm, right of = dadatz] (dadat){ };
			\node[minimum size = 3mm,draw,circle,  inner sep =1pt,node distance = 7mm, right of = eatz] (eat){ };
			\node[minimum size = 2mm,draw,circle,  inner sep =1pt,node distance = 12mm, left of = ea] (ead){ };
			\node[minimum size = 2mm,draw,circle,  inner sep =1pt,node distance = 7mm, above of = ead] (eada){ };
			\node[minimum size = 2mm,draw,circle,  inner sep =1pt,node distance = 12mm, left of = dada] (dadad){ };
			\node[minimum size = 2mm,draw,circle,  inner sep =1pt,node distance = 4mm, above of = dadad] (dadada){ };
			\draw (ea) edge (eatz);
			\draw (eatz) edge (eat);
			\draw (ea) edge (ead);
			\draw (dada) edge (dadad);
			\draw (dadad) edge (dadada);		
			\draw (eada) edge (ead);
		
			\draw (dadda) edge (dadad);
			\draw (ea) edge (dad);
			\draw (dada) ..controls +(-.4,-.63) and +(-.4,.63).. (e);
			\draw (dadada) ..controls +(-.5,-.63) and +(-.5,.63).. (f);
			
			\draw (da) ..controls +(-.2,-.63) and +(-.2,.63).. (d);
			\draw (dadatz) ..controls +(-.2,-.63) and +(-.2,.63).. (eatz);
			\draw (da) ..controls +(.2,-.63) and +(.2,.63).. (d);
			\draw (dadatz) ..controls +(.2,-.63) and +(.2,.63).. (eatz);

		\draw (a) edge (b);
		\draw (b) edge (c);
		\draw (d) edge (a);
		\draw (d) edge (e);
		\draw (e) edge (f);
		\draw (fa) edge (f);
		\draw (c) edge +(.4,0);
				\draw (datz) edge (da);
				\draw (dat) edge (datz);
					\draw (dad) edge (da);
					\draw (dad) edge (dadd);
					\draw (eada) edge (dadd);
					\draw (dada) edge (dadatz);
					\draw (dadatz) edge (dadat);
					\draw (dadd) edge (dadda);
			\draw (e) edge (ea);
			\draw (f) edge +(-.4,0);
			\draw (dat) edge +(.2,0);
			
			\draw (fa) edge (ead);
			
			\draw (dad) edge (dada);
						draw (f) edge +(0,-.2);

		\draw  (a) ..controls +(-.2,-.9) and +(.2,-.9).. (a);%
		\draw (b) ..controls +(-.2,-.9) and +(.2,-.9).. (b);
		\draw (c) ..controls +(-.2,-.9) and +(.2,-.9).. (c);
		\draw (dat) ..controls +(-.2,-.9) and +(.2,-.9).. (dat);
		\draw (datz) ..controls +(-.2,-.9) and +(.2,-.9).. (datz);
		\draw (eat) ..controls +(-.14,-.63) and +(.14,-.63).. (eat);
		\draw (dadat) ..controls +(-.14,-.63) and +(.14,-.63).. (dadat);

		\node[node distance = 10mm,circle,  inner sep =2pt, right of = c]{$\cdots$};
		\node[node distance = 3mm,circle,  inner sep =2pt, left of = dadd]{\tiny $\cdots$};
		\node[node distance = 6mm,circle,  inner sep =2pt, right of = dat]{\footnotesize $\cdots$};
		\node[node distance = 10mm,circle,  inner sep =2pt, left of = f]{$\cdots$};
			\node[node distance = 3mm,circle,  inner sep =2pt, left of = dadda]{\tiny $\cdots$};
			\node[node distance = 3mm,circle,  inner sep =2pt, right of = dadda]{\tiny $\cdots$};
			
				\node[node distance = 3mm,circle,  inner sep =2pt, left of = dadada]{\tiny $\cdots$};
				\node[node distance = 3mm,circle,  inner sep =2pt, right of = dadada]{\tiny $\cdots$};
					\node[node distance = 3mm,circle,  inner sep =2pt, left of = eada]{\tiny $\cdots$};
					\node[node distance = 3mm,circle,  inner sep =2pt, right of = eada]{\tiny $\cdots$};
						\node[node distance = 3mm,circle,  inner sep =2pt, left of = fa]{\tiny $\cdots$};
						\node[node distance = 3mm,circle,  inner sep =2pt, right of = fa]{\tiny $\cdots$};
							\node[node distance = 3mm,circle,  inner sep =2pt, left of = dadad]{\tiny $\cdots$};
							\node[node distance = 3.3mm,circle,  inner sep =2pt, right of = eat]{\tiny $\cdots$};
							\node[node distance = 3.3mm,circle,  inner sep =2pt, right of = dadat]{\tiny $\cdots$};
					\node[node distance = 3mm,circle,  inner sep =2pt, left of = ead]{\tiny $\cdots$};

		\end{tikzpicture}
		\caption{The Schreier graph $\GG(\BS{1}{2}, \gen{a}, \oneset{a,\ov a, t,\ov t})$. As long as the random walk stays in the right part, the probability to return to the origin is not exponentially decreasing.}\label{fig:SchreierBS} \vspace{-3mm}
	\end{center}
\end{figure}
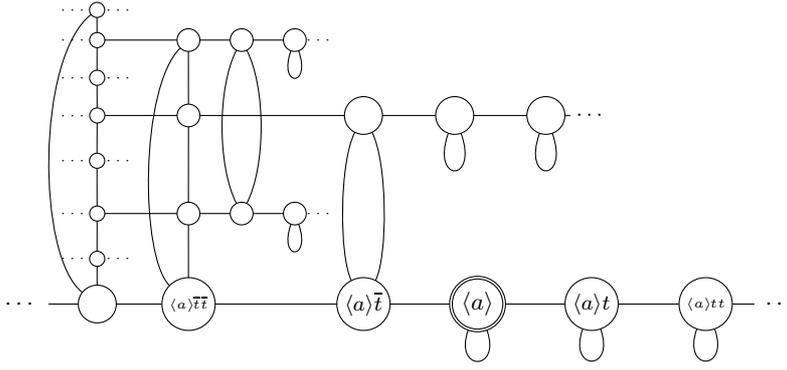
 $\Diamond$\end{example}

\begin{proof}[Proof of \prettyref{thm:h}]
The case $A=H$ leads to an amenable Schreier graph as we calculated in \prettyref{rem:hnnamen}.

For the other direction, let $[H:A]\geq 2$ and $[H:B]\geq 2$. 
We show condition \ref{gromov} in \prettyref{prop:amenability}. In order to do so, we define a function $f:\lquot{H}{G}\to \lquot{H}{G}$ as follows: We fix $1\neq c\in C$ and $1\neq d \in D$. 
 (Note that we cannot exclude $c = d$.) 
 Let $w=x_0t^{\eps_1}x_1 \cdots t^{\eps_k}x_k$ be in normal form according to \prref{lem:hnnNF}.
 If $w$ ends with $x_k=1$, then there is some $\gam \in \oneset{c,d}$ (depending on $\eps_k$) such that both $w \gam\stl$ and $w\gam\ov \stl$ are in normal form (we identify $1 \in H$ with the empty word); and we set $f(Hw \gam \stl) = f(Hw\gam\ov \stl) = Hw$.
 If $w$ ends with $x_k\neq 1$, then $w\stl c$ and $w\ov \stl d$ are both in normal form, and we set $f(Hw \stl c) = f(Hw\ov \stl d) = Hw$. Furthermore, we set $f(Hw') = Hw'$ for words $w'$ in normal form which are not of the above form. Because of \prettyref{lem:hnnNF}, the function $f$ is well-defined.
 Let $k\in \N$ some number which is large enough that 
 $c$, $d$, $\stl$ and $\ov \stl$ can we written with at most $k$ letters from $\Sig$. Then we obtain $\sup \set{d(f(Hw),H w)}{Hw \in \lquot{H}{G}}\leq 2k<\infty$. 
Moreover, we have $\abs{f^{-1}(Hw)}\geq 2 $ for all $w\in G$.
\end{proof}

\subsection{Proof of \prettyref{thm:one}}
\label{sec:amengeneric}
 The aim of this section is to prove \prettyref{thm:one}
 using \prettyref{thm:a} for amalgams and \prettyref{thm:h} for HNN extensions. In the following we say a word over a set of generators is 
 {\em elliptic} if it represents an elliptic element in the 
 generated group. This is the complement of the words representing hyperbolic group elements. Therefore, \prettyref{thm:one}
is a consequence of the following two assertions: 
\begin{enumerate}
\item The set of elliptic words is strongly negligible in $\Sig^*$.
\label{giaa}
\item The set of cyclically reduced, elliptic words is strong\-ly negligible in $\Del$ (recall that $\Del$ denotes the set of cyclically reduced words).
\label{gibb}
\end{enumerate}
We only state the proof of \prettyref{thm:one} for HNN extensions. 
The proof for amalgamated products is an almost verbatim translation of the 
proof for HNN extensions. Details are left to the reader.

First, let us consider another finite set of generators $\Sig'$ with some additional properties. Without restriction we may assume that, first, $\Sig'$ is symmetric and, second, 
$\y,\ov \y \in \Sig'$ and $\Sig'\sse_G H \cup \oneset{\y,\ov \y}$. 
Part of the difficulty in the proof is due to the 
 fact that encoding $\Sig^*$ into $\Sig'^*$ turns $\Sig^*$ into a strongly negligible set inside $\Sig'^*$, in general. This makes the proof a little bit tedious. 

Recall that we do not require $\Sig$ to be symmetric.
 Every letter $a \in \Sig$ can be written as a reduced word 
 $u_{a} \in \Sig'^*$. We need also the somehow other direction.
 For every letter $a\in \Sig$ and every 
 factorization $u_{a}= p_{a}q_{a}$ (with $q_{a}\neq 1$) we assign a pair of (reduced) words $(P_{a},Q_{a}) \in \Sig^*\times \Sig^*$ with $P_a =_G p_a$ and $Q_a =_G q_a$. We let $c \in \N $ be a constant with
$\abs{P_{a}Q_{a}} \leq c$ and  $\abs{u_a}\leq c$ for all $a \in \Sig$ and all factorizations  $u_{a}= p_{a}q_{a}$.
We obtain a list $\cL= \set{(P_{a},Q_{a},a)}{a \in \Sig}$ containing at most $c\abs \Sig$ tuples.

\begin{proof}[Proof of \ref{giaa}] 
The idea is now to label certain words in $\Sig^*$ such that 
at the end all elliptic words are labeled and that the set of words having a label is strongly negligible.

By \prref{thm:h}, the Schreier graph $\GG(G,H,\Sig\cup \ov \Sig)$ is non-amenable (where $\ov \Sig$ is a disjoint set of formal inverses for $\Sig$). 
Thus, by \prref{prop:indgen}, the random walk on $\GG(G,H,\Sig)$ has exponentially decreasing return probability. 
Therefore, we can label all words in $\Sig^*$ which represent elements 
in $H$~-- now the set of words with a label is strongly negligible (see \prettyref{rem:generic}).
But of course, there are many more elliptic elements. Hence, we have to label more words. 

Next, for every $(P_{a},Q_{a},a) \in \cL$ and every word of the form 
$Q_{a}v P_{a}$ which has a label, we label the word $w =av$, too. 
The length of words $Q_{a}v P_{a}$ is between $\abs w$ and 
$\abs w+c$ and for every $v$ there are at most $c\cdot \abs \Sig$ of the form $Q_{a}vP_{a}$. Hence, the set of labeled words remains strongly negligible.

In one more round we label all words $w_{2}w_{1}$ where $w_{1}w_{2}$
has a label. Still the set of words with label is strongly neg\-ligible.

We claim that all elliptic elements are labeled now. To see this, 
let $w=a_{1}\cdots a_{n}$ be a word with $a_{i} \in \Sig$ such that 
$w$ is conjugate to some element in $H$. Then the same is true 
for the word $w'=u_{a_{1}}\cdots u_{a_{n}} \in \Sig'^*$ because
$w =_{G} w'$. Recall that we have $\Sig' \sse_G H \cup\os{t,\ov t}$.
Hence, by Collins' Lemma, we can factorize 
$w'=w'_{1}w'_{2}$ such that $w'_{2}w'_{1}$ represents an element in $H$. 
Therefore, we find some index $i$ and a factorization 
$u_{a_{i}}
= p_{a_{i}}q_{a_{i}}$ with $q_{a_i}\neq 1$ such that 
\begin{align*}
 w'_{2}w'_{1}&= q_{a_{i}}u_{a_{i+1}} \cdots u_{a_{n}} u_{a_{1}} \cdots u_{a_{i-1}}p_{a_{i}}\\
 &=_{G} Q_{a_{i}}{a_{i+1}} \cdots{a_{n}} {a_{1}} \cdots {a_{i-1}}P_{a_{i}} = Q_{a_{i}}v P_{a_{i}}\in_G H.
\end{align*}
The word $Q_{a_{i}}v P_{a_{i}}$ got a label because it is a word representing an element in $H$. 
This puts a label on $a_{i}v$, too. Finally, $w$ is a transposition of 
$a_{i}v$. Hence, latest in the last step, $w$ got a label. 
\end{proof}

\noindent{\bf Proof of \ref{gibb}.}
Recall that $\Del$ denotes the subset of cyclically reduced words 
and that $\Sig$ is assumed to be symmetric (otherwise we did not define the notion of reduced word). 
Let $\Del'$ denote the subset of reduced words. Hence, 
we have $$\Del \sse \Del' \sse \Sig^*.$$
We have $\Sig \geq 4$ (otherwise $G$ is a cyclic group); hence, the degree $d$ of the Schreier graph $ \GG(G,H,\Sig)$ is at least $4$. 
Moreover, \begin{align*}
\qquad\qquad\Del' \cap \Sig^n &= d(d-1)^{n-1} \qquad\qquad\qquad\text{and}\\
\Del \cap \Sig^n &\geq d(d-2)(d-1)^{n-2}
\end{align*} for all $n \geq 1$.
Since $\frac{\abs{\Del' \cap \Sig^n}}{\abs{\Del \cap \Sig^n}} \leq \frac{d-1}{d-2}$, the proof of \ref{gibb} is reduced to show the following lemma.
\begin{lemma}\label{lem:gicc} 
The set of cyclically reduced, elliptic words is strongly negligible in the set of reduced words $\Del'$.
\end{lemma}

\begin{proof}
Again we label words in $\Del'$ such that at the end all elliptic words are labeled and that the set of words with label is strongly negligible.

 We start to label all reduced words in $\Sig^*$ which have length at most $4c$. Next, for each $(P_{a},Q_{a},a)\in \cL$ we consider all words
 of the form $W=Q_{a}v P_{a}$ where, first, $av$ is reduced, 
 second, $\abs v \geq 2c$, and third, 
 $Q_{a}v P_{a}$ represents an element in $H$. In particular, $W$ is elliptic.
 Since $v$ is reduced and long enough, we obtain a reduced word 
 $W' = Q_{a}'u P_{a}'$ such that $W=_{G} W'$, 
 $Q_{a}'$ is a prefix of $Q_{a}$, and $P_{a}'$ is a suffix of $P_{a}$ (we assumed $P_a$, $Q_a$ to be reduced).
 We label $W'$. It represents an element in $H$. Thus, the set of reduced words 
 with a label is strongly negligible in $\Del'$ by \prref{prop:amenability} \ref{randomwwbt}. 
 
 In the next round we label all words $av$ where there exists 
 some $(P_{a},Q_{a},a)\in \cL$ such that $Q_{a}v P_{a}$ reduces to 
 some reduced word 
 $W'$ which is labeled. Now, we can recover $av$ if we know the following four data:
 \begin{align}\label{eq:data} 
W',\;(P_{a},Q_{a},a)\in \cL, \;\abs{Q_{a}'}, \text{ and } \abs{P_{a}'}. 
\end{align}
Thus, each labeled $W'$ can produce at most $c\abs{\Sig}(c+1)^2 \in \Oh(1)$ new labels on words of length between $\abs{W'} -2c$ and  $\abs{W'} + 2c +1$. Hence, the set of reduced words with label is still 
strongly negligible. Note that now all words representing elements of $H$ are labeled.

 Finally, as above, we label all words $w_{2}w_{1}$ where $w_{1}w_{2}$
has a label; and the set of reduced words with label remains
strongly negligible. 

It remains to show that all cyclically reduced elliptic words are labeled. This is true for all words of length at most $4c$. 
Hence, 
let $n > 4c$ and $w=a_{1}\cdots a_{n}$ be a cyclically reduced word with $a_{i} \in \Sig$ such that 
$w$ is conjugate to some element in $H$. As above we switch to the corresponding word $w'=u_{a_{1}}\cdots u_{a_{n}} \in \Sig'^*$. By Collins' Lemma, we can factorize 
$w'=w'_{1}w'_{2}$ such that $w'_{2}w'_{1}\in_G H$. 
Again, we find some index $i$ and $(P_{a_{i}},Q_{a_{i}},a_{i})\in \cL$
together with a factorization $w'_{2}w'_{1}= q_{a_{i}}u_{a_{i+1}} \cdots u_{a_{n}} u_{a_{1}} \cdots u_{a_{i-1}}p_{a_{i}}$ and $u_{a_{i}}
= p_{a_{i}}q_{a_{i}}$.
Now, since $w$ is cyclically reduced, the 
word $a_{i}v = a_{i}a_{i+1} \cdots{a_{n}} {a_{1}} \cdots {a_{i-1}}$ is reduced. Moreover,
 $$w'_{2}w'_{1}=_{G} Q_{a_{i}}{a_{i+1}} \cdots{a_{n}} {a_{1}} \cdots {a_{i-1}}P_{a_{i}} = Q_{a_{i}}v P_{a_{i}}\in_G H.$$
According to our procedure the reduced word $a_{i}v$ has got a label. Finally, $w$ is a transposition of 
$a_{i}v$. Hence, latest in the last step, $w$ got a label. 
\end{proof}


\end{document}